\documentclass{amsart}
\title{On the uniqueness of Yangians}
\author[S. Gautam]{Sachin Gautam}
\address{Department of Mathematics, The Ohio State University}
\email{gautam.42@osu.edu}
\author[C. Wendlandt]{Curtis Wendlandt}
\address{Department of Mathematics and Statistics, University of Saskatchewan}
\email{wendlandt@math.usask.ca}

\author[S. Xu]{Siwei Xu}
\address{Department of Mathematics, The Ohio State University}
\email{xu.4877@osu.edu}

 
\usepackage{amsmath,amssymb,amsthm}
\usepackage[dvipsnames]{xcolor} % For coloured fonts 
\usepackage[
	colorlinks,
	plainpages,
	citecolor=red!50!black,
    filecolor=Darkgreen,
    linkcolor=blue!40!black,
    urlcolor=cyan!50!black!90]{hyperref} % Setup for hyperlinks
\usepackage[all]{xy}
\usepackage{graphicx}
\usepackage{enumitem}
\newtheorem*{thm}{Theorem}
\newtheorem*{lem}{Lemma}
\newtheorem*{cor}{Corollary}

\newtheorem*{prop}{Proposition}

\let\oldproofname=\proofname
\renewcommand{\proofname}{\upshape\textsc{\oldproofname}}

\theoremstyle{definition}

\newtheorem*{rem}{Remark}
\newtheorem*{ex}{Example}

\numberwithin{equation}{section}

\newcommand{\lp}{\left(}
\newcommand{\rp}{\right)}

\newcommand{\g}{\mathfrak{g}}
\newcommand{\h}{\mathfrak{h}}

\newcommand{\Lsl}{\mathfrak{sl}}

\newcommand{\T}{\mathcal{T}}
\newcommand{\U}{\mathcal{U}}
\newcommand{\A}{\mathcal{A}}
\newcommand{\CC}{\mathcal{C}}

\newcommand{\K}{\mathcal{K}}

\newcommand{\J}{\mathcal{J}}

\newcommand{\C}{\mathbb{C}}

\newcommand{\N}{\mathbb{N}}

\newcommand{\Z}{\mathbb{Z}}

\newcommand {\ul}[1]{\underline{#1}}

\newcommand{\End}{\operatorname{End}}
\newcommand{\Hom}{\operatorname{Hom}}

\newcommand {\aand}{\qquad\text{and}\qquad}

\newcommand {\Aut}{\operatorname{Aut}}

\newcommand{\ds}{\displaystyle}

\newcommand {\Omit}[1]{}

\newcommand {\ad}{\operatorname{ad}}

\newcommand {\Yhg}{Y_\hbar(\g)}

\newcommand{\iso}{\xrightarrow{\,\smash{\raisebox{-0.5ex}{\ensuremath{\scriptstyle\sim}}}\,}}

\newcommand {\bfI}{{\mathbf I}}

\renewcommand {\H}{\mathcal H}

\newcommand {\Id}{\operatorname{Id}}

\renewcommand {\sl}{\mathfrak{sl}}
\newcommand {\curtiscomment}[1]{
}
\newcommand{\BB}[1]{\left[#1\right]}

\begin{document}

\begin{abstract}
Let $\mathfrak{g}$ be a simple Lie algebra over the complex numbers, and let $\mathfrak{g}[u]$ denote its polynomial current algebra. In the mid--1980s, Drinfeld introduced the Yangian of $\mathfrak{g}$ as the unique solution to a quantization problem for a natural Lie bialgebra structure on $\mathfrak{g}[u]$. More precisely, Theorem 2 of [Dokl. Akad. Nauk SSSR \textbf{283} (1985), no.~5, 1060--1064] asserts that $\mathfrak{g}[u]$ admits a unique homogeneous quantization --- the Yangian of $\mathfrak{g}$ --- which is described explicitly via generators and relations, starting from a copy of $\mathfrak{g}$ and its adjoint representation.  Although the representation theory of Yangians has since undergone substantial development, a complete proof of Drinfeld's theorem has not appeared. 
In this article, we present a proof of the assertion that $\mathfrak{g}[u]$ admits at most one homogeneous quantization. Our argument combines cohomological and computational methods, and outputs a presentation of any such quantization using Drinfeld's generators and a reduced set of defining relations.
\end{abstract}

\maketitle

\setcounter{tocdepth}{1}
\tableofcontents

\setlength{\parskip}{5pt}

\section{Introduction}\label{sec:intro}
%======================================

\subsection{}\label{ssec:intro-summary}
%------------
Let $\g$ be a simple Lie algebra over the field of complex numbers $\C$, and let $\g[u]$ be the Lie algebra of polynomial maps $\C\to \g$, with Lie bracket defined pointwise. The latter Lie algebra, called the (polynomial) current algebra of $\g$, is naturally graded and admits a Lie bialgebra structure with cobracket 
\begin{equation*}
\delta:\g[u]\to \g[u]\otimes \g[u]
\end{equation*}
which is obtained by using the residue and Killing forms on $\C[u^{\pm 1}]$ and $\g$, respectively, to identify the graded dual of $\g[u]$ with $u^{-1}\g[u^{-1}]$, and then setting $\delta$ to be the transpose of the natural Lie bracket on $u^{-1}\g[u^{-1}]$; see Section \ref{ssec:current}.

In Theorem 2 of Drinfeld's foundational paper \cite{Dr}, it is asserted that the Lie bialgebra $(\g[u],\delta)$ admits a unique \textit{homogeneous quantization}. That is, up to isomorphism, there is a unique graded Hopf algebra over the graded ring $\C[\hbar]$ that provides a flat deformation of the enveloping algebra $U(\g[u])$, and whose coproduct $\Delta$ recovers $\delta$ as its semiclassical limit; see Definition \ref{ssec:hgs-q-defn}.
Moreover, this unique quantization --- which Drinfeld called the \textit{Yangian} of $\g$ ---  is explicitly presented in this same theorem using generators and relations.
 The main goal of this article is to make available a proof of the uniqueness assertion from Drinfeld's theorem which, at the same time, explains how to derive a presentation of any such quantization.

  \subsection{Drinfeld's theorem}
%------------

In order to motivate and precisely formulate our main results, we begin by recalling the statement of Drinfeld's theorem. To this end,  let $\{x_\lambda\}_{\lambda=1}^{\dim \g}\subset \g$ be an orthonormal basis of $\g$ relative to a fixed non-degenerate, invariant, and symmetric bilinear form $(\cdot,\cdot)$ on $\g$.  Let $c_{\lambda\mu}^\nu=([x_\lambda,x_\mu],x_\nu)$ be the structure constants of $\g$ in this basis, so that 
$
[x_\lambda,x_\mu]=\sum_{\nu}c_{\lambda\mu}^\nu x_\nu.
$
The following theorem is a restatement of Theorem 2 from \cite{Dr}.
\begin{thm}\label{thm:Dr}
The Lie bialgebra $(\g[u],\delta)$ admits a unique homogeneous quantization
$\A$. As a unital associative $\C[\hbar]$-algebra, $\A$ is generated by elements $I_\lambda$
and $J_\lambda$, for $1\leq \lambda\leq\dim(\g)$, with defining relations
\begin{gather}\label{ieq:1}
[I_\lambda,I_\mu] = \sum_\nu c_{\lambda\mu}^\nu I_\nu,\qquad
[I_\lambda,J_\mu] = \sum_\nu c_{\lambda\mu}^\nu J_\nu,
\\ 
\begin{aligned}\label{ieq:2}
[J_\lambda,[J_\mu,I_\nu]] - [I_\lambda,[J_\mu,J_\nu]] &= \hbar^2 S_0(\lambda,\mu,\nu),\\
[[J_\lambda,J_\mu],[I_r,J_s]] - [[J_r,J_s],[I_\lambda,J_\mu]] &= \hbar^2 S_1(\lambda,\mu,r,s).
\end{aligned}
\end{gather}
Moreover, the grading on $\A$ is given by $\deg(I_\lambda)=0$ and $\deg(J_\lambda)=1$, while its coproduct $\Delta$ is uniquely determined by
\begin{equation}\label{ieq:3}
\Delta(I_\lambda) = \square(I_\lambda) \quad \text{ and }\quad
\Delta(J_\lambda) = \square(J_\lambda) + \frac{\hbar}{2} \sum_{\mu,\nu}
c_{\lambda\mu}^\nu I_\nu\otimes I_\mu,
\end{equation}
where we have set $\square(x) = x\otimes 1 + 1\otimes x$.
\end{thm}
Note that the first equation of \eqref{ieq:1}  says exactly that the linear map $\iota:\g\to \A_0$ given by $\iota(x_\lambda)=I_\lambda$ is a Lie algebra homomorphism, and the second equation then says that $J:\g\to \A_1$ given by $J(x_\lambda)=J_\lambda$ is a $\g$-module intertwiner. 
 Thus, \eqref{ieq:1}
is equivalent to the assertion that $\A$ contains a copy of $U(\g)$ in degree $0$,
and a copy of the adjoint representation of $\g$ in degree $1$.

In \eqref{ieq:2}, $S_0(\lambda,\mu,\nu)$ and $S_1(\lambda,\mu,r,s)$ are certain explicitly given elements of
degrees $0$ and $1$ in $\A$, respectively. Their exact expressions will not be relevant to us, and hence have been omitted. In addition, we note that, according to Drinfeld \cite{Dr}, the first equation of \eqref{ieq:2} is redundant when $\g\cong \sl_2$, while the second is redundant when $\g\ncong \sl_2$. It is also stated in Example 3 of \cite[\S6]{DrQG} that both of these equations arise from the relations 	of \eqref{ieq:1} and that requirement that \eqref{ieq:3} determines an algebra homomorphism; see also the discussion preceding Definition 12.1.2 in \cite{CPBook}.

\subsection{Main results}\label{ssec:intro-main-thm}
%------------
As indicated in Section \ref{ssec:intro-summary}, the graded Hopf algebra $\A$ from Theorem \ref{thm:Dr} (or its specialization
at $\hbar=1$) is the well-known Yangian of $\g$, whose structure and
representation theory have been extensively studied over the last forty years --- we refer the reader to \cite[Ch.~12]{CPBook} for a survey of some of the foundational results in this area. However, to the best of our knowledge, a full proof
of Drinfeld's theorem has never appeared in the literature. 

In this article, we provide a proof of the uniqueness  statement of Theorem \ref{thm:Dr} which simultaneously  establishes that any homogeneous quantization of 
$(\g[u],\delta)$ must admit a presentation as in Theorem \ref{thm:Dr}, but with the relations from \eqref{ieq:2} replaced by a smaller set of equations. 
To state this precisely, let $\h$ be a Cartan subalgebra of $\g$ and, for each $h\in \h$, define
\begin{equation*}
\nu(h) = \frac{1}{2} \sum_{\alpha>0} \alpha(h)x^-_\alpha x^+_\alpha \in U(\g),
\end{equation*}
where the sum is taken over a choice of set of positive roots for $\g$ relative to $\h$, and $x^\pm_\alpha$ are root vectors for $\pm \alpha$ with $(x^+_\alpha,x^-_\alpha)=1$; see Sections \ref{ssec:Lie} and \ref{ssec:nu}. The following is the main result of this article. 

\begin{thm}\label{thm:main}
Let $\H$ be a homogeneous quantization of $(\g[u],\delta)$.
Then, as a unital associative $\C[\hbar]$-algebra, $\H$ is generated by elements $\iota(x)$ and $J(x)$, for $x\in\g$, subject to 
the following relations:

\begin{enumerate}
\item\label{main-1} $\iota:\g\to \H$ is a Lie algebra homomorphism and $J:\g\to \H$ is a $\g$-module homomorphism. That is, they are $\C$-linear and, for each $x,y\in \g$, one has
\begin{equation*}
\iota([x,y]) = [\iota(x),\iota(y)] \quad \text{ and }\quad J([x,y])=[\iota(x),J(y)].
\end{equation*}

\item\label{main-2} If $\g\ncong\sl_2$, then for each $h_1,h_2\in \h$, one has 
\begin{equation*}
[J(h_1),J(h_2)] =\hbar^2\iota([\nu(h_2),\nu(h_1)]).
\end{equation*}

\item\label{main-3}  If $\g\cong \sl_2$ and  $e,f,h$ form a fixed $\sl_2$-triple in $\g$, then
\begin{equation*}
\left[[J(e), J(f)], J(h)\right] = \hbar^2 (\iota(f)J(e)-J(f)\iota(e))\iota(h).
\end{equation*}

\end{enumerate}
Moreover, the grading on $\H$ is given by $\deg(\iota(x))=0$ and $\deg(J(x))=1$, while its coproduct $\Delta$ is uniquely determined by
\begin{equation}\label{ieq:main-D}
\Delta(\iota(x)) = \square(\iota(x)) \quad \text{ and }\quad
\Delta(J(x)) = \square(J(x)) + \frac{\hbar}{2} [\iota(x)\otimes 1,\Omega_\iota],
\end{equation}
where $\Omega_\iota=(\iota\otimes\iota)(\Omega)$ with $\Omega\in\mathrm{Sym}^2(\g)^\g$ the Casimir tensor of $\g$.
\end{thm}
Here it is understood that the Lie algebra homomorphism $\iota:\g\to \H$ from \eqref{main-1} has been extended to a $\C$-algebra homomorphism $U(\g)\to \H$ in \eqref{main-2}. 

Although not needed in this article, we note that the relations of \eqref{main-2} and \eqref{main-3} are known to follow from the equations of Theorem \ref{thm:Dr}. For instance, this has been established in the proof of Theorem 2.6 in \cite{GRWEquiv}; we refer the reader to (2.20) therein and Step 1 of \cite[\S A]{GRWEquiv} for further details. 

\subsection{Outline of proof}\label{ssec:intro-proof}
%-----------
Our proof of Theorem \ref{thm:main} is divided into two main parts. In the first part, we establish 
that if $\H$ is any homogeneous quantization of $(\g[u],\delta)$, then there 
exists a $\g$-module homomorphism $J:\g\to\H_1$ as in \eqref{main-1} of Theorem \ref{thm:main}, which satisfies the coproduct formula \eqref{ieq:main-D}. This result is stated in Theorem \ref{thm:gens}, which provides the main conclusion of Section \ref{sec:hgs-gens}. Its proof is based on a cohomological argument
using Whitehead's first lemma and Cartier's computation of the cohomology of the coalgebra complex associated to $U(\g)$; see Sections \ref{ssec:coalg} and \ref{ssec:CE}. 

The second part is computational, and focuses on establishing that the relations of Theorem \ref{thm:main} necessarily hold in $\H$. This is achieved in Section \ref{sec:hgs-rels} by proving that the difference between the left and right-hand sides of the identities in \eqref{main-2} and \eqref{main-3}  must be primitive elements in $\H$, and therefore belong to $\hbar^2\h$ and $\hbar^3\h$, respectively, by Proposition \ref{pr:prim}. These differences are then shown to be identically zero in Sections \ref{ssec:Nsl2-3} and \ref{ssec:sl2-2}. We refer the reader to Theorem \ref{thm:relns}, which provides the main result of Section \ref{sec:hgs-rels}, for further details. 

These two main ingredients are then combined in Section \ref{sec:hgs-unique} to give a proof of Theorem \ref{thm:main}, which is restated in an equivalent form in Theorem \ref{thm:uniqueness}. Its proof also hinges on the fact that the relations \eqref{main-1}--\eqref{main-3} of Theorem \ref{thm:main} deform a set of defining relations for the current algebra $\g[u]$, which is a consequence of Theorem \ref{thm:current-min}.

\subsection{Remarks}\label{ssec:intro-remarks}
%------------
We end this introduction by emphasizing that Theorem \ref{thm:main} does \textit{not} claim to provide a proof of the existence of a homogeneous quantization of $(\g[u],\delta)$, as it does not assert that the algebra defined by the relations \eqref{main-1}--\eqref{main-3} is such a quantization, with coproduct given by \eqref{ieq:main-D}. However, this is in fact the case, and already follows from existing results in the literature on Yangians. Crucially, the graded algebra  defined by the relations of the theorem is nothing but the Yangian $\Yhg$, expressed in the minimalistic presentation obtained in \cite[Thm.~2.13]{GNW} by the second author in joint work with Guay and Nakajima, based on an earlier work of Levendorskii \cite{Lev-Gen}. In Section \ref{ssec:existence}, we will elaborate on this point and briefly survey the results that collectively imply that the Yangian, presented as in Theorem \ref{thm:main} and \cite{GNW}, is a homogeneous quantization of $(\g[u],\delta)$.

\iffalse
We end this introduction by remarking that the presentation
of the homogeneous quantization given in Theorem \ref{thm:main}
above is nothing but the minimalistic presentation obtained
by the second author, joint with Guay and Nakajima in
\cite[Thm.~2.13]{GNW}, based on an earlier work of Levendorskii
\cite{Lev-Gen}. In these papers, this minimalistic presentation
is shown to be isomorphic to Drinfeld's {\em new or loop}
presentation \cite{DrNew}. In the loop presentation, the PBW
theorem for Yangians is known \cite{LevPBW}. The implication of
these results is that the algebra defined in Theorem \ref{thm:main}
is a homogeneous quantization of $(\g[u],\delta)$, something
we don't prove (or need) here. We only show that {\em if there is
a Yangian}, it has to be isomorphic to the one stated in
Theorem \ref{thm:main}.
\fi

\subsection{Acknowledgments}\label{ssec:acknowledge}
%------------

The second author gratefully acknowledges the support of the Natural Sciences and Engineering Research Council of Canada (NSERC), provided via the Discovery Grants Program (Grant RGPIN-2022-03298 and DGECR-2022-00440).

\section{Simple Lie algebras and current algebras}\label{sec:defns}
%=============================================================

In this preliminary section, we recall the key properties of the classical structures underpinning this article: the Lie algebra $\g$ and its polynomial current algebra $\g[u]$. 

\subsection{Simple Lie algebras}\label{ssec:Lie}
%-----------------------------------------------
Let $\g$ be a finite-dimensional simple Lie algebra over $\mathbb{C}$. 
Let $\h \subset \g$ be a Cartan subalgebra and let $R\subset\h^*\setminus\{0\}$ be the set of roots associated to
the pair $(\g,\h)$. We will write $\g_\alpha\subset \g$ for the root space associated to an arbitrary root $\alpha\in R$. Let $\{\alpha_i\}_{i\in \bfI} \subset R$ be a base of simple roots, and let $R_+\subset R$
be the corresponding set of positive roots. We fix a non-degenerate, invariant, symmetric bilinear form $(\cdot ,\cdot)$ on $\g$ and, for each $i\in \bfI$, choose elements $x_i^{\pm}\in \g_{\pm \alpha_i}$ satisfying $(x_i^+,x_i^-)=1$. Note that the Cartan elements $\{t_i\}_{i\in \bfI}\subset \h$ defined by 
\[
t_i:=[x_i^{+},x_i^{-}]
\]
then satisfy $(h,t_i) = \alpha_i(h)$ for each $h\in \h$. In the case where $\g\cong\mathfrak{sl}_2$, we will simply write $e=x_i^+$, $f=x_i^-$ and $h=t_i$, where $\bfI=\{i\}$. 

Let $\Omega \in \mathrm{Sym}^2(\g)^{\g}$ denote the Casimir tensor of the bilinear form $(\cdot,\cdot)$, so that
\[
\Omega = \sum_{i\in\bfI} t_i\otimes \varpi_i^{\vee} + \sum_{\alpha\in R_+} (x^+_\alpha\otimes x^-_{\alpha}
+ x^-_{\alpha}\otimes x^+_\alpha),
\]
where $\{\varpi_i^{\vee}\}_{i\in\bfI}\subset\h^*$ is the basis of fundamental coweights, and
$x^{\pm}_{\alpha}\in \g_{\pm\alpha}$ are chosen so that $(x^+_\alpha, x^-_\alpha)=1$.

\subsection{The linear map $\nu:\h\to U(\g)^{\h}$}\label{ssec:nu}
%------------------------------------------------------------
For later purposes, we introduce a linear map $\nu:\h\to U(\g)^{\h}$ by setting
\[
\nu(h) := \frac{1}{2} \sum_{\alpha\in R_+} \alpha(h) x^-_\alpha x^+_\alpha \quad \forall\; h\in \h.
\]
In addition, we define auxiliary elements $\{w_i^{\pm}\}_{i\in \bfI}\subset U(\g)$ by the formula
\begin{equation*}
 w_i^{\pm} := \pm\frac{1}{(\alpha_i,\alpha_i)} [\nu(t_i),x_i^\pm] \quad \forall \; i\in \bfI.
\end{equation*}
The main properties of these elements, stated below, were proved in
\cite{GNW}; see Lemma 3.9 of {\em loc. cit.} for \eqref{GNW:1} and \eqref{GNW:2},
and \S 4.4 (pages 890--892) for \eqref{GNW:3}.
\begin{lem}\label{lem:GNW}
The following relations hold in $U(\g)$.

\begin{enumerate}
\item\label{GNW:1} For every $h\in\h$ and $i\in\bfI$, we have
$[\nu(h),x_i^{\pm}] = \pm \alpha_i(h) w_i^{\pm}$.

\item\label{GNW:2} For every $i,j\in\bfI$, we have
$[w_i^+,x_j^-] = [x_i^+,w_j^-] = \delta_{ij} \lp \nu(t_i) - \frac{1}{2} t_i^2 \rp$.

\item\label{GNW:3} For every $h_1,h_2\in\h$, we have
%\begin{equation}\label{eq:Delta-nu}
\[
\Delta([\nu(h_1),\nu(h_2)]) = \square \lp[\nu(h_1),\nu(h_2)]\rp - \frac{1}{4}
\left[\left[ h_1\otimes 1, \Omega\right], \left[ h_2\otimes 1, \Omega\right]\right]
\]
%\end{equation}
where $\square(y) = y\otimes 1 + 1\otimes y$ and  $\Delta:U(\g)\to
U(\g)^{\otimes 2}$ is the coproduct on $U(\g)$, uniquely determined by $\Delta(x) = \square(x)$ for all $x\in \g$.
\end{enumerate}
\end{lem}

\subsection{The polynomial current Lie bialgebra}\label{ssec:current}
%--------------------------------------------------------------

Let $\g[u]=\g\otimes \C[u]$ denote the Lie algebra of polynomials in a single variable $u$ with coefficients in $\g$. 
This is an $\mathbb{N}$-graded Lie algebra over $\C$,
with $n$-th homogeneous component $\g u^n=\g\otimes \C u^n$ and Lie bracket determined by
\begin{equation*}
[x u^n,y u^m]=[x,y] u^{n+m} \quad \forall\; x,y\in \g \; \text{ and }\; n,m\geq 0.
\end{equation*}
Moreover, the formula 
\begin{equation*}
\delta(xu^n)=\sum_{a+b=n-1} [x\otimes 1, \Omega]u^a v^b\in (\g\otimes \g)[u,v]\cong \g[u]\otimes \g[u],
\end{equation*}
for all $x\in \g$ and $n\geq 0$, defines a degree $-1$ Lie cobracket $\delta$ on $\g[u]$, endowing it with the structure of 
an $\mathbb{N}$-graded Lie bialgebra. Equivalently, the above formula defines a graded linear map $\delta:\g[u]\to \g[u]\wedge \g[u]$  of degree $-1$ which satisfies the cocycle and co-Jacobi identities 
\begin{gather*}
\delta([f(u),g(u)])=[\delta(f(u)),\Delta(g(u))]+[\Delta(f(u)),\delta(g(u))], \\ 
(\Id +(1\,2\,3)+(1\,3\,2))\circ (\delta\otimes \Id) \circ \delta=0, 
\end{gather*}
for all $f(u),g(u)\in \g[u]$, where in the second relation the symmetric group on three letters acts on $\g[u]^{\otimes 3}$ by permuting its tensor factors.  

This Lie bialgebra structure arises naturally from the residue form on the loop algebra $\g[u,u^{-1}]$ given by
\begin{equation*}
\langle xu^n, yu^m\rangle = -(x,y)\mathrm{Res}_t(t^{m+n})=-\delta_{m+n,-1}(x,y)
\end{equation*}
for all $x,y\in \g$ and $n,m\in \Z$. This form identifies the Lie subalgebra $u^{-1}\g[u^{-1}]\subset \g[u,u^{-1}]$ with the graded dual of $\g[u]$, and the Lie cobracket $\delta$ on $\g[u]$ described above is just the transpose of the Lie bracket on $u^{-1}\g[u^{-1}]$. We refer the reader to \cite[\S3]{DrQG} or \cite[\S2.5]{WRQD}, for instance, for further details. In addition, we note that the Lie cobracket $\delta$ can be expressed in terms of the Casimir tensor $\Omega$ of $\g$ by
\begin{equation*}
\delta(f)(u,v)=\left[f(u)\otimes 1 + 1\otimes f(v), \frac{\Omega}{u-v}\right] \quad \forall \quad f(u)\in \g[u].
\end{equation*}

\subsection{A minimal presentation of \texorpdfstring{$\g[u]$}{g[u]}}\label{ssec:current-min}
%--------------------------------------------------------------

To prove the main result of this article, we will make use of a presentation of $\g[u]$ given in terms of its degree $0$ and degree $1$ generators. In this subsection, we record  this realization of $\g[u]$. 

Let $\mathfrak{a}$ be the Lie algebra over $\C$ with generators  $\imath(x)$ and $G(x)$, for each $x\in \g$, and the following defining relations:
\begin{enumerate}
\item\label{g[u]-min:1} $\imath: \g \to \mathfrak{a}$ is a Lie algebra map and $G: \g \to \mathfrak{a}$ is a $\g$-module homomorphism.
That is, for all $x,y\in\g$ and $\lambda,\mu\in\C$, we have:
\begin{equation*}
\begin{aligned}
\imath(\lambda x + \mu y) = \lambda \imath(x) + \mu \imath(y), &\quad& \imath([x,y]) = [\imath(x),\imath(y)], \\
G(\lambda x + \mu y) = \lambda G(x) + \mu G(y), &\quad& G([x,y]) = [\imath(x),G(y)].
\end{aligned}
\end{equation*}

\item\label{g[u]-min:2} For each $i,j\in \bfI$, one has
\begin{align*}
[G(t_i),G(t_j)]&=0 &&\text{ if }\quad \g\ncong\sl_2 \\
\left[[G(e), G(f)], G(h)\right] &= 0  &&\text{ if }\quad \g\cong\sl_2
\end{align*}
\end{enumerate}
It is clear from this definition that $\mathfrak{a}$ admits an $\N$-graded Lie algebra structure, with grading uniquely determined by $\deg \imath(x)=0$ and $\deg G(x)=1$ for all $x\in \g$.  Moreover, the above definition is such that the relations imposed on $\imath(x)$ and $G(x)$ are satisfied by the elements $x$ and $xu$ in $\g[u]$, respectively. In fact, we have the following theorem.

\begin{thm}\label{thm:current-min}
The assignment $\psi$ defined by
\begin{equation*}
\psi(\imath(x))= x \quad \text{ and }\quad \psi(G(x))=xu \quad \forall\; x\in \g
\end{equation*}
uniquely extends to an isomorphism of $\N$-graded Lie algebras $\psi:\mathfrak{a}\to\g[u]$. 
\end{thm}
\begin{rem}
When $\g\cong \sl_2$, the presentation of $\g[u]$ described by this theorem (namely, that provided by $\mathfrak{a}$) is exactly that obtained from the presentation of the Yangian $\Yhg$ of $\g$ established in Theorem  1.2 of \cite{Lev-Gen} by reducing modulo $\hbar$ and using that $\Yhg/\hbar \Yhg\cong U(\g[u])$. For $\g\ncong \sl_2$, this same observation holds with \cite[Thm.~1.2]{Lev-Gen}  replaced by Theorem 2.13 of \cite{GNW}, which provides a strengthening of the former result under the assumption that $\g$ is not of rank $1$; see also Sections \ref{ssec:yangian} and \ref{ssec:existence}. Thus, Theorem \ref{thm:current-min} is a consequence of the results of \cite{Lev-Gen,GNW}. However, it is worth pointing out that these results for $\Yhg$ are much more complicated than Theorem \ref{thm:current-min} --- the arguments of \cite{Lev-Gen} and \cite{GNW} become much simpler when $\hbar$ is specialized to zero. 
\end{rem}

Henceforth, we shall identify $\mathfrak{a}$ and $\g[u]$ without further explanation. Note that, with respect to this realization, the Lie cobracket $\delta$ on $\g[u]$ from Section \ref{ssec:current} satisfies 
\begin{equation*}
\delta(\imath(x))=0 \quad \text{ and }\quad \delta(G(x))=[x\otimes 1,\Omega] \quad \forall\; x\in \g.
\end{equation*}

\section{Homogeneous quantization: generators}\label{sec:hgs-gens}
%===================================================

The goal of this section is twofold. First, we review the definition and basic properties of homogeneous quantizations of graded Lie bialgebras, with particular focus on the Lie bialgebra $(\g[u],\delta)$; see Sections \ref{ssec:hgs-q-defn} and \ref{ssec:prim}. Second, we carry out the first part of the proof of Theorem \ref{thm:main}, as outlined in Section \ref{ssec:intro-proof}. 

The main result of the section is Theorem \ref{thm:gens}, which establishes that any homogeneous quantization $\H$ of the Lie bialgebra $(\g[u],\delta)$ admits a $\g$-module homomorphism $\J:\g\to \H$ that satisfies the coproduct relation \eqref{ieq:main-D} for $J$, and has image in its degree one component.

\subsection{Definition}\label{ssec:hgs-q-defn}
%----------------------------------------------
A \textit{homogeneous quantization} of the Lie bialgebra $(\g[u],\delta)$ is an 
$\mathbb{N}$-graded Hopf algebra $\H=\bigoplus_{n\geq 0} \H_n$ over the graded ring $\C[\hbar]$, where $\deg(\hbar)=1$, 
satisfying the following properties:
\begin{enumerate}[label=($\H$\arabic*), ref=$\H$\arabic*]\setlength{\itemsep}{3pt}
\item\label{N-quant:1} $\H$ is torsion free as a $\C[\hbar]$-module.
\item\label{N-quant:2} $\H / \hbar \H$ is isomorphic to $U(\g[u])$ as a graded Hopf algebra.
\item\label{N-quant:3} The cobracket $\delta$ is related to the coproduct $\Delta$ on $\H$ by the formula
\begin{equation*}
\delta(x)=\frac{\Delta(\widetilde{x})-\Delta^{\mathrm{21}}(\widetilde{x})}{\hbar} \mod \hbar \H\otimes\H
\end{equation*}
for all $x\in \g[u]$, where $\widetilde{x}\in \H$ is any element satisfying $\widetilde{x} \equiv x \mod \hbar$.
\end{enumerate}

\begin{rem}\label{R:Hom-quant} More generally, a homogeneous quantization of an $\N$-graded Lie bialgebra $(\mathfrak{b},\delta_\mathfrak{b})$ is an $\N$-graded Hopf algebra $\H$ over $\C[\hbar]$ satisfying the axioms \eqref{N-quant:1}--\eqref{N-quant:3} with $\g[u]$ replaced by $\mathfrak{b}$; see \cite[\S2.4]{WRQD}. Such a quantization always has the following properties:
\begin{enumerate}\setlength{\itemsep}{3pt}
\item\label{Hom-quant:1} $\H$ is isomorphic to $U(\mathfrak{b})[\hbar]$ as an $\mathbb{N}$-graded $\C[\hbar]$-module; see Corollary 2.6 of \cite{WRQD}. 

\item\label{Hom-quant:2} The $\hbar$-adic completion of $\H$ is a homogeneous quantization of $(\mathfrak{b},\delta_\mathfrak{b})$ over $\C[\![\hbar]\!]$ in the sense first defined in the context of Yangians in \cite{Dr}; see \cite[Def.~2.12]{WRQD}. Moreover, any such quantization over $\C[\![\hbar]\!]$ arises in this way.

\item\label{Hom-quant:3} If $\mathfrak{b}_k$ denotes the $k$-th homogeneous component of $\mathfrak{b}$ and $\pi:\H\twoheadrightarrow U(\mathfrak{b})$ is the quotient map provided by \eqref{N-quant:2}, then there is a unique 
embedding of graded Hopf algebras 
\begin{equation*}
\jmath:U(\mathfrak{b}_0)[\hbar]\hookrightarrow \H
\end{equation*}
for which $\pi\circ \jmath|_{U(\mathfrak{b}_0)}$ is the natural inclusion $U(\mathfrak{b}_0)\hookrightarrow U(\mathfrak{b})$. Indeed, by \eqref{N-quant:2} one has
$
\H_0=(\H / \hbar \H)_0  \cong U(\mathfrak{b})_0 = U(\mathfrak{b}_0)
$
as Hopf algebras over $\C$. This implies that there is a unique homomorphism $\jmath:U(\mathfrak{b}_0)[\hbar]\to \H$ of graded Hopf algebras over $\C[\hbar]$ with $\pi\circ \jmath|_{U(\mathfrak{b}_0)}$ equal to the inclusion of $U(\mathfrak{b}_0)$ into $U(\mathfrak{b})$. Since $\H$ is torsion free, $\jmath$ is necessarily injective. 
\end{enumerate}
\end{rem}

\subsection{Primitive elements}\label{ssec:prim}
%----------------------------------------------

Now let $\H$ be a homogeneous quantization of $(\g[u],\delta)$. Then 
by \eqref{Hom-quant:3} of Remark \ref{R:Hom-quant}, there is an embedding of graded Hopf algebras 
\begin{equation*}
U(\g)[\hbar]\hookrightarrow \H
\end{equation*}
which identifies the enveloping algebra $U(\g)$ with the component $\H_0$ of $\H$. Henceforth, we shall use this embedding to view $U(\g)[\hbar]$ as a subalgebra of  $\H$. In particular, $ \g\otimes_\C\C[\hbar]\cong\C[\hbar]\cdot \g\subset U(\g)[\hbar]$ is a subset of $\H$. This fact is used in the following proposition, which computes the Lie algebra of primitive elements in $\H$.
\begin{prop}
\label{pr:prim}
Let $\H$ be a homogeneous quantization of $(\g[u],\delta)$. Then the Lie algebra $\operatorname{Prim}(\H)$ of primitive elements in $\H$ is equal to $\g \otimes_{\C} \C[\hbar]$:
\begin{equation*}
\operatorname{Prim}(\H) := \left\{y \in \H: \Delta(y)=\square(y)\right\}=\g \otimes_{\C} \C[\hbar].
\end{equation*}
\end{prop}
\begin{proof} This follows by a modification of the argument used to establish the analogous assertion for a specific filtered quantization of $(\g[u],\delta)$ --- the Yangian of $\g$ defined over $\C$ --- in Lemma B.1 of \cite{GTLW19}.
We will prove by induction on $k$ that
$\ds \operatorname{Prim}(\H) \cap \H_k = \g \otimes_{\C} \C \hbar^k$.
The base case follows since $\H_0=U(\g)$ and $\operatorname{Prim}(U(\g)) = \g$. 
Assuming the assertion for $k$, let 
$y \in \H_{k+1}$ be an arbitrary primitive element. Let $\bar{y}$ be its image in $\H / \hbar \H \cong U(\g[u])$. Thus,
$\bar{y}$ is a primitive element of degree $k+1$ in $U(\g[u])$, hence $\bar{y} = x  u^{k+1}$ for some $x \in \g$. 
Using axiom \eqref{N-quant:3} and that $\Delta(y)=\Delta^{21}(y)$, we obtain
\begin{equation*}
\left. \delta(x u^{k+1})=\frac{\Delta(y)-\Delta^{21}(y)}{\hbar} \right\vert_{\hbar = 0} = 0.
\end{equation*}
By definition of $\delta$ (see Section \ref{ssec:current}), this means $[x\otimes 1,\Omega]=0$, which implies that $x=0$ and hence $\bar{y}=0$.
Thus, $y = \hbar z$ for some $z\in \operatorname{Prim}(\H)\cap \H_k$, and we are done by induction. \qedhere
\end{proof}

\subsection{Generators of a homogeneous quantization}\label{ssec:gens}
%-----------------------------------------------------------------------

Since any homogeneous quantization $\H$ of $(\g[u],\delta)$ contains $U(\g)$ as its degree $0$ component,  each component $\H_k$ of $\H$ becomes a $\g$-module under the adjoint action $x\cdot y=[x,y]$, for all $x\in \g$ and $y\in \H_k$. The following theorem, which provides the main result of Section \ref{sec:hgs-gens}, shows that there is a distinguished $\g$-module homomorphism $\J:\g\to \H_1$ which is compatible with $\Delta$, where $\g$ acts on itself via the adjoint action.  
\begin{thm}\label{thm:gens}
Let $\H$ be a homogeneous quantization of $(\g[u],\delta)$. Then there exists a $\g$-module homomorphism $\J: \g \to \H_1$ satisfying $\J(x)=xu$ modulo $\hbar$ and 
\begin{equation*}
\Delta(\J(x)) = \J(x) \otimes 1 + 1 \otimes \J(x) + \frac{\hbar}{2}[x \otimes 1, \Omega]
\end{equation*}
for each $x\in \g$.
Moreover, if $\J^{'}: \g \to \H_1$ is another such map, then there is $\lambda \in \mathbb{C}$ such that
$\J(x) - \J^{'}(x) = \hbar \lambda x$ for all $x\in \g$.
\end{thm}

\begin{proof}
Let us first prove the uniqueness assertion. Let $\J$ and $\J^{'}$ be two maps satisfying the conditions of the theorem. Since they both have image in $\H_1$ and are equal modulo $\hbar$, we have
$\J(x) - \J^{'}(x) \in \hbar U(\g)$ for all $x\in \g$. As $\H$ is torsion free, it follows that there is a unique linear map $B:\g\to U(\g)$ satisfying
\begin{equation*}
\J(x)- \J^{'}(x)=\hbar B(x) \quad \forall\; x\in \g. 
\end{equation*}
Since each difference $\J(x) - \J^{'}(x)$ is a primitive element, the image of $B$ is contained in the Lie algebra $\g$ of primitive elements in $U(\g)$. Moreover, as $\J$ and $\J^{'}$ are both $\g$-module homomorphisms, we must have $B\in \End_\g(\g)$. Thus, by Schur's lemma, there is $\lambda\in \C$ such that $B=\lambda \Id_\g$, and hence 	
$\J(x) - \J^{'}(x) = \hbar \lambda x$ for each $x\in\g$.

To prove the existence of $\J$, we begin by choosing an arbitrary linear map 
\begin{equation*}
f:\g\to \H_1
\end{equation*}
whose composite with the quotient map $\H_1\to \H_1 / \hbar \H_0 \cong U(\g[u])_1$ equals the linear map $G:\g\to \g u\subset U(\g[u])_1$, given by $G(x)=xu$ for each $x\in \g$. That is, $f$ satisfies $f(x)= xu \mod \hbar \H_0$ for each $x\in \g$. Now define
\begin{equation*}
\gamma: \g \times \g \to U(\g) \aand \eta: \g \to U(\g) \otimes U(\g)
\end{equation*}
by the following equations:
\begin{align}
        f([x,y]) &= [x,f(y)] + \hbar \gamma(x,y) \label{eq:gamma} \\
        \Delta(f(x)) &= f(x) \otimes 1 + 1 \otimes f(x) + \frac{\hbar}{2}[x \otimes 1, \Omega] + \hbar\eta(x) \label{eq:eta}
\end{align}

\noindent {\bf Claim.} There exists a linear map $\varphi: \g \to U(\g)$ satisfying
\begin{equation*}
\gamma(x,y) = [x,\varphi(y)] - \varphi([x,y])  \aand  \eta(x) = \square(\varphi(x)) - \Delta(\varphi(x))
\end{equation*}
for all $x,y\in \g$. 

Assuming such a $\varphi$ exists, it is easy to see that $\J(x) := f(x) + \hbar \varphi(x)$ 
satisfies the properties listed in the theorem. Indeed, for each $x,y\in \g$, we have 
\begin{equation*}
    \begin{aligned}
        \J([x,y]) &= f([x,y]) + \hbar \varphi([x,y])\\
        & = [x,f(y)] + \hbar \gamma(x,y) + \hbar([x,\varphi(y)] - \gamma(x,y)) \\ 
        &=
 [x,f(y) + \hbar\varphi(y)] \\
 &= [x,\J(y)],
 \end{aligned}
 \end{equation*}
 and hence $\J\in \Hom_\g(\g,\H_1)$. Similarly, for each $x\in \g$, one has
 \begin{equation*}
 \begin{aligned}
        \Delta(\J(x)) 
        &= \Delta(f(x) + \hbar \varphi(x))\\
        & = \square(f(x)) + \frac{\hbar}{2}[x \otimes 1, \Omega] + \hbar \eta(x) + \hbar(\square(\varphi(x)) - \eta(x)) \\
       % &= \square(f(x) + \hbar \varphi(x)) + \frac{\hbar}{2}[x \otimes 1, \Omega]\\
        & = \square(\J(x)) + \frac{\hbar}{2}[x \otimes 1, \Omega].
    \end{aligned}
\end{equation*}

The existence of $\varphi$ satisfying the conditions of the claim follows from a cohomological argument, which we carry out in the remainder of this section. 
In Section \ref{ssec:gamma-eta}, we use the axioms of a bialgebra to obtain equations satisfied by $\gamma$ and $\eta$.
We view these equations as defining a cocycle in a double cochain complex, introduced in
Section \ref{ssec:double-complex}.
In Sections \ref{ssec:coalg} and \ref{ssec:CE}, 
we relate our double complex to the well-known coalgebra and Chevalley--Eilenberg complexes. We review the
standard results (Cartier's theorem and Whitehead's lemma) about the cohomology of these complexes and, in
Section \ref{ssec:pf-claim}, we apply them to prove the claim. \qedhere
\end{proof}
\begin{rem}\label{rem:Hopf}
Let $\J$ be as in the statement of Theorem \ref{thm:gens}, and let $x\in \g$. Then the values of the counit $\varepsilon$ and antipode $S$ of  $\H$  on $\J(x)$ are given by
\begin{equation*}
\varepsilon(\J(x))=0 \quad \text{ and }\quad S(\J(x))=-\J(x)+\frac{\hbar}{4}c_\g x,
\end{equation*}
where $c_\g$ is the eigenvalue of the quadratic Casimir element $C\in U(\g)$ on the adjoint representation of $\g$. Indeed, since $\varepsilon$ is an algebra homomorphism and $\J$ is a $\g$-module homomorphism, we have
\begin{equation*}
\varepsilon(\J([x,y]))=\varepsilon([x,\J(y)])=[\varepsilon(x),\varepsilon(\J(y))]=0 \quad \forall \; x,y\in \g.
\end{equation*}
Since $\g$ is a perfect Lie algebra, it follows that $\varepsilon(\J(x))=0$ for all $x\in \g$. Similarly, as $S$ is the convolution inverse of the identity map $\Id_\H$, we have 
\begin{equation*}
0=\varepsilon(\J(x))=(m \circ (S\otimes \Id_\H) \circ \Delta)(\J(x))=S(\J(x))+\J(x)+\frac{\hbar}{2} m([\Omega, x\otimes 1]),
\end{equation*}
where $m$ is the product on $\H$. The claimed formula for $S(\J(x))$ then follows from the observation that $m([\Omega, x\otimes 1])=-\frac{1}{2}\rho_{\mathrm{ad}}(C)(x)$, where $\rho_{\mathrm{ad}}:U(\g)\to \End(\g)$ is the action homomorphism for the adjoint of $\g$ on itself.
\end{rem}

\subsection{Properties of $\gamma$ and $\eta$}\label{ssec:gamma-eta}
%----------------------------------------------------------------

In what follows, $\gamma:\g\times \g \to U(\g)$ and $\eta:\g\to U(\g)\otimes U(\g)$ are the functions defined by the equations \eqref{eq:gamma} and \eqref{eq:eta}, respectively.

\begin{prop}\label{pr:gamma-eta}
The maps $\gamma$ and $\eta$ satisfy the equations
\begin{align}
\gamma([x,y],z) - \gamma(x,[y,z]) + \gamma(y,[x,z]) - [x,\gamma(y,z)] + [y,\gamma(x,z)] &=0 \label{eq:gg} \\
1\otimes \eta(x) - (\Delta\otimes\Id) (\eta(x)) + (\Id\otimes\Delta) (\eta(x)) - \eta(x)\otimes 1 &= 0 \label{eq:ee} \\
\eta([x,y]) - [\square(x),\eta(y)] + \square(\gamma(x,y)) - \Delta(\gamma(x,y)) &=0 \label{eq:eg}
\end{align}
for each $x,y,z\in \g$. Moreover, $\eta$ is symmetric:  $\eta(x) = \eta^{21}(x)$ for all $x\in \g$. 
\end{prop}

\begin{proof}
The equations \eqref{eq:gg}--\eqref{eq:eg} are consequences of the Jacobi identity, the coassociativity of $\Delta$, and the fact that $\Delta$
is an algebra homomorphism,  respectively. We begin by proving \eqref{eq:gg}. 
Let $x,y,z \in \g$ and consider the following relation in $\H$:
\begin{equation*}
[[x,y],f(z)] = [x,[y,f(z)]] - [y,[x,f(z)]]
\end{equation*}
By definition of $\gamma$, the three terms of this equation can be rewritten as follows:
\begin{align*}
[[x,y],f(z)] &=  f([[x,y],z]) - \hbar \gamma([x,y],z)\\
[x,[y,f(z)]] &=  f([x,[y,z]]) - \hbar \gamma(x,[y,z]) - \hbar [x,\gamma(y,z)]\\
[y,[x,f(z)]] &=  f([y,[x,z]]) - \hbar \gamma(y,[x,z]) - \hbar [y,\gamma(x,z)]
\end{align*}
Substituting these expressions back into the original identity and applying the linearity of $f$ and the Jacobi identity for $\g$ yields the relation \eqref{eq:gg}.

Next we prove \eqref{eq:ee}. Since $\Delta$ is coassociative, we have
\begin{equation*}
(\Delta \otimes \Id)  \Delta(f(x)) = (\Id \otimes \Delta)  \Delta(f(x))
\end{equation*}
for each $x\in \g$. 
From the definition \eqref{eq:eta} of $\eta$, the left-hand side of this equation expands as
\begin{equation*}
\begin{aligned}
(\Delta \otimes \Id)  \Delta(f(x)) ={} & \Delta(f(x)) \otimes 1 + 1\otimes 1 \otimes f(x) 
+ \hbar (\Delta \otimes \Id) \eta(x)  \\
& + \frac{\hbar}{2}[x \otimes 1 \otimes 1 + 1\otimes x \otimes 1, \Omega_{13} + \Omega_{23}] \\
 ={} & \square_3(f(x)) + \frac{\hbar}{2} [x \otimes 1 \otimes 1, \Omega_{12}+\Omega_{13}] + \hbar \, \eta(x) \otimes 1 \\
& + \frac{\hbar}{2}[1 \otimes x \otimes 1, \Omega_{23}] + \hbar (\Delta \otimes \Id) \eta(x) 
\end{aligned}
\end{equation*}
where $\square_3(a) = a \otimes 1 \otimes 1 + 1 \otimes a \otimes 1 + 1 \otimes 1 \otimes a$.
In a similar fashion, we have
\begin{equation*}
\begin{aligned}
(\Id\otimes \Delta) \Delta(f(x))
={}&\square_3(f(x)) + \frac{\hbar}{2} [1 \otimes x \otimes 1, \Omega_{23}] + \hbar \, 1 \otimes \eta(x) \\
&+ \frac{\hbar}{2}[x \otimes 1 \otimes 1, \Omega_{12}+\Omega_{13}] + \hbar (\Id \otimes \Delta) \eta(x).    
\end{aligned}
\end{equation*}
Equating both of these expressions and simplifying outputs the identity \eqref{eq:ee}.

Let us now turn to \eqref{eq:eg}. As $\Delta$ is an algebra homomorphism, we have
\begin{equation*}
\Delta([x,f(y)]) = [\Delta(x),\Delta(f(y))]
\end{equation*}
for every $x,y\in \g$. Expanding both sides using \eqref{eq:gamma} and \eqref{eq:eta} gives
\begin{align*}
\BB{\Delta(x),\Delta(f(y))} &= \BB{\square(x), \, \square(f(y)) + \frac{\hbar}{2}[y \otimes 1, \Omega] + \hbar \eta(y)}\\
&= \square (f([x, y])) - \hbar \, \square(\gamma(x,y)) + \frac{\hbar}{2}[[x,y]\otimes 1, \Omega] + \hbar [\square(x), \eta(y)], \\
\Delta([x,f(y)])  &= \Delta (f([x,y]) - \hbar \gamma(x,y)) \\
&= \square (f([x, y])) + \frac{\hbar}{2}[[x,y]\otimes 1, \Omega] + \hbar \, \eta([x,y]) - \hbar \Delta(\gamma(x,y)).
\end{align*}
The equation \eqref{eq:eg} now follows by reinserting these expressions into the identity $\Delta([x,f(y)]) = [\Delta(x),\Delta(f(y))]$ and simplifying.

To complete the proof of the proposition, it remains to establish that $\eta(x)=\eta^{21}(x)$ for each $x\in \g$.  To this end, note that since  $f:\g\to \H_1$ satisfies $f(x)=G(x)\mod \hbar \H_1$, the axiom \eqref{N-quant:3} for a homogeneous quantization of $(\g[u],\delta)$ yields
\begin{equation*}
\Delta(f(x)) - \Delta^{21}(f(x)) = \hbar \delta(G(x))=\hbar[x\otimes 1,\Omega]. 
\end{equation*}
On the other hand, the defining equation \eqref{eq:eta} for $\eta$ gives
\begin{equation*}
    \begin{aligned}
        \Delta(f(x)) - \Delta^{21}(f(x)) &= \frac{\hbar}{2}([x \otimes 1, \Omega]-[1 \otimes x, \Omega]) + \hbar (\eta(x) - \eta^{21}(x)) \\
        &= \hbar \delta(G(x)) + \hbar (\eta(x) - \eta^{21}(x))
    \end{aligned}
\end{equation*}
and hence we must indeed have $\eta(x) = \eta^{21}(x)$. \qedhere
\end{proof}

\subsection{The double complex}\label{ssec:double-complex}
%-----------------------------------------------------
We now define a bicomplex $\ds \left\{\K^{m,n}\right\}_{m\geq 0, n\geq 1}$ equipped with a horizontal
differential $\partial_H : \K^{m,n}\to \K^{m+1,n}$ and a vertical differential $\partial_V : \K^{m,n}\to
\K^{m,n+1}$ as follows. For each $m\geq 0$ and $n\geq 1$, set
\[
\K^{m,n} := \Hom(\wedge^m(\g) \otimes \g_{\ad}, U(\g)^{\otimes n}),
\]
where $\g_{\ad}$ denotes the adjoint representation of $\g$.
For $\omega\in\K^{m,n}$, the horizontal differential $\partial_H(\omega)$ is given by the formula
\begin{equation}\label{eq:del-CE}
\begin{aligned}
\partial_H(\omega)(x_1,&\ldots,x_{m+1};v) \\
={} & \sum_{1\leq i<j\leq m+1} (-1)^{i+j}
\omega([x_i,x_j],\ldots,\hat{x_i},\ldots,\hat{x_j},\ldots x_{m+1} ; v) \\
&+ \sum_{i=1}^{m+1} (-1)^{i-1}\Big(
x_i\cdot \omega(x_1,\ldots,\hat{x_i},\ldots x_{m+1}; v)\\[-12pt]
& \hspace{13em}
 - 
\omega(x_1,\ldots,\hat{x_i},\ldots x_{m+1}; x_i\cdot v)
\Big)
\end{aligned}
\end{equation}
for all $x_1,\ldots,x_{m+1}\in \g$ and $v\in \g_{\ad}$, where we have used the notation $\hat{y}$ to indicate
that the variable $y$ is omitted from the argument. By a little abuse of notation, $x\cdot m$ means
the action of $x\in \g$ on an element $m$ of a $\g$--module ($\g_{\ad}$ or $U(\g)^{\otimes n}$).

The vertical differential  is defined by $\partial_V(\omega)(\ul{x};v) = \delta_n(\omega(\ul{x};v))$,
where $\delta_n : U(\g)^{\otimes n} \to U(\g)^{\otimes n+1}$ is given as follows:
\begin{equation}\label{eq:del-C}
\begin{aligned}
\delta_n(y) &= 1\otimes y + \sum_{i=1}^n (-1)^i \lp \Id^{\otimes i-1}\otimes\Delta\otimes\Id^{n-i}\rp (y)  + (-1)^{n+1} y\otimes 1
\end{aligned}
\end{equation}
It is a routine exercise to verify that $\partial_H$ and $\partial_V$ are commuting differentials,
which we leave to the reader.

\noindent\textbf{Remarks.} Let us gather a few important preliminary observations about the above  bicomplex:
\begin{enumerate}\setlength{\itemsep}{3pt}
\item If we keep $m$ fixed, the resulting cochain complex $(\K^{m,\bullet},\partial_V)$ is nothing
but $\Hom(\wedge^m(\g)\otimes\g_{\ad},-)$ applied to the coalgebra complex $(T^{\bullet}(U(\g)),\delta_\bullet)$ --- also called the \textit{cobar complex} ---
whose cohomology was computed by Cartier in \cite{Car57}. We review Cartier's theorem below
in Section \ref{ssec:coalg}.

\item\label{R-double-comlex:2} Keeping $n$ fixed, the cochain complex $(\K^{\bullet,n},\partial_H)$ is isomorphic, via
Hom-tensor adjointness, to the Chevalley--Eilenberg complex valued in the $\g$-module
$\g_{\ad}^*\otimes U(\g)^{\otimes n}$. The computation of its first and second cohomology groups, called
Whitehead's lemma, is recalled in Section \ref{ssec:CE}.

\item\label{R-double-comlex:3} The equations for $\gamma\in\K^{1,1}$ and $\eta\in\K^{0,2}$ obtained in Proposition \ref{pr:gamma-eta}
can be expressed equivalently as 
\begin{equation*}
\partial_H(\gamma) = 0,\qquad \partial_V(\eta)=0, \qquad \partial_V(\gamma) = \partial_H(\eta),
\end{equation*}
and $\eta = \eta^{21}$. 
Moreover, the claim from the proof of Theorem \ref{thm:gens} which we aim to prove (see below \eqref{eq:eta}) is equivalent to the assertion that there is $\varphi\in\Hom(\g_{\ad},U(\g)) = \K^{0,1}$ satisfying $\partial_H(\varphi) = \gamma$ and $\partial_V(\varphi)=\eta$; see Figure \ref{fig:cartoon} below.
\begin{figure}[h]
\[
\xy
(0,0)*{\K^{0,1}}="A"; (-5,-3)*{\text{\rotatebox[origin=c]{45}{$\in$}}}; (-7,-5)*{\varphi};
(20,0)*{\K^{1,1}}="B"; (20,-3)*{\text{\rotatebox[origin=c]{90}{$\in$}}}; (20,-6)*{\gamma};
(0,20)*{\K^{0,2}}="C"; (-7,19)*{\eta\in};
(20,20)*{\K^{1,2}}="D";
(40,0)*{\K^{2,1}}="E";
(0,40)*{\K^{0,3}}="F";
{\ar_{\partial_H} "A"; "B"};
{\ar_{\partial_H} "B"; "E"};
{\ar_{\partial_V} "A"; "C"};
{\ar_{\partial_V} "C"; "F"};
{\ar_{\partial_V} "B"; "D"};
{\ar_{\partial_H} "C"; "D"};
\endxy
\]
\caption{Bicomplex $\K^{\bullet,\bullet}$}\label{fig:cartoon}
\end{figure}

\item One can generalize this bicomplex $\K^{\bullet,\bullet}$ by replacing $\g_{\ad}$
by an arbitrary $\g$-module $V$, and $U(\g)$ by $\mathrm{Sym}(W)$ where $W$ is another
$\g$-module. Recall that $U(\g)$ and $\mathrm{Sym}(\g)$ are isomorphic as coalgebras
with $\g$-action.
\end{enumerate}

\subsection{Cartier's theorem}\label{ssec:coalg}
%------------------------------------------------
Let $(C,\Delta,\varepsilon)$ be a cocommutative coalgebra over an arbitrary field $k$.
The coalgebra complex associated to $C$ is $(T^{\bullet}(C),\delta_\bullet)$, where $T^n(C)=C^{\otimes n}$
 and $\delta_n : T^n(C) \to T^{n+1}(C)$ is given by \eqref{eq:del-C} above. Define
an involution $\sigma \in \Aut(T^n(C))$ by the formula
\[
\sigma(x_1\otimes\cdots\otimes x_n) = (-1)^{\frac{n(n+1)}{2}} x_n\otimes\cdots\otimes x_1.
\]
Then, $\sigma$ commutes with $\delta_\bullet$ and the complex decomposes
into $\pm 1$ eigenspaces $T^\bullet(C) = T^\bullet_+(C) \oplus T^\bullet_-(C)$, where
\begin{equation*}
 T^n_\pm(C)=\{y\in T^n(C): \sigma(y)=\pm y\}.
\end{equation*}
The following theorem is a corollary of a result established by Cartier in \cite{Car57} which identifies the cohomology of the complex $(T^{\bullet}(\mathrm{Sym}(V)),\delta_\bullet)$ with the exterior algebra $\Lambda(V)$. This theorem may also be found in \cite[Prop.~3.11]{Dr-quasi} and \cite[Thm.~XVIII.7.1]{KasBook95}, where two alternative proofs of Cartier's result are given.
\begin{thm}\label{thm:cartier}
Let $C = (\mathrm{Sym}(V), \Delta, \epsilon)$, where $V$ is a finite dimensional vector space over $k$, and $\Delta$ and $\epsilon$ are the algebra homomorphisms determined by
\begin{equation*}
\Delta(v) = \square(v) \aand \epsilon(v) = 0
\end{equation*}
for every $v\in V$. Then $H^{2n}(T^\bullet_-(C),\delta_\bullet)=0$ for each $n\geq 0$. 
\end{thm}

Below, we state the specific corollary of this theorem that is relevant to us. By applying equation~(6.2) from~\cite[XVIII.6.2]{KasBook95} and using the fact that $\delta_n$ is a $\g$-intertwiner, we deduce that
$\Hom(\g_{\ad}, T^\bullet_-(U(\g)))$
decomposes as a direct sum of the following subcomplexes (with differentials omitted for brevity):
\[
\Hom_{\g}(\g_{\ad}, T^\bullet_-(U(\g))) \quad \text{and} \quad \g  \Hom(\g_{\ad}, T^\bullet_-(U(\g))).
\]
Since $\Hom(\g_{\ad}, -)$ is an exact functor, it follows that $\Hom_{\g}(\g_{\ad}, -)$ is exact as well. The corollary then becomes an immediate
consequence of the $n=1$ instance of Cartier's theorem.
\begin{cor}\label{cor:vertical}
$H^2(\Hom_{\g}(\g_{\ad},T^\bullet_-(U(\g)),\partial_V) = 0$.
\end{cor}
%
%
%\begin{proof}
%Take $n=1$ of the theorem above, combined with the fact that $\Hom(\g_{\ad}, -)$ is an exact functor, we have $$H^2(\Hom(\g_{\ad},T^\bullet_-(U(\g)),\partial_V) = 0$$
%Also, by the fact that $\delta_n$ is a $\g$-intertwiner, we can apply (6.2) of \cite[XVIII]{KasBook95}, and get
%$$\Hom(\g_{\ad},T^\bullet_-(U(\g)))=\Hom_{\g}(\g_{\ad},T^\bullet_-(U(\g)))\oplus \g\Hom(\g_{\ad},T^\bullet_-(U(\g)))$$
%Therefore, $H^2(\Hom_{\g}(\g_{\ad},T^\bullet_-(U(\g)),\partial_V) = 0$. 
%\end{proof}

\subsection{Chevalley--Eilenberg complex}\label{ssec:CE}
%------------------------------------------------------

Let $M$ be an arbitrary $\g$-module. Recall that the Chevalley--Eilenberg complex  associated to the pair $(\g,M)$ is defined as follows. For each $n>0$, we let $\CC^n(\g,M) = \Hom(\wedge^n \g, M)$ denote the space of all 
alternating $n$-linear maps from $\g$ to $M$, and we set $\CC^0(\g,M) := M$.  The differential 
$d_{CE} : \CC^n(\g,M) \to \CC^{n+1}(\g,M)$ is given by the same formula as \eqref{eq:del-CE}, except now
the second term has only a $\g$-action on $M$:
\begin{equation*}
    \begin{aligned}
        d_{CE}(f)(x_1,\ldots,x_{n+1}) ={}& 
\sum_{1 \le i<j \le n+1} (-1)^{i+j} \omega([x_i,x_j],\ldots,\hat{x_i},\ldots,\hat{x_j},\ldots,x_{n+1}) \\
&+ \sum_{i=1}^{n+1}(-1)^{i-1}x_i \cdot \omega(x_1,\ldots,\hat{x_i},\ldots,x_{n+1})\\
    \end{aligned}
\end{equation*}
This data defines a cochain complex with $n$-th cohomology group denoted $H^n(\g, M)$. 
The following result, known as Whitehead's lemma, asserts that these groups are trivial when $M$ is finite-dimensional and $n=1$ or $2$. We refer the reader to 
\cite[\S 18.3]{KasBook95} for a proof.

\begin{thm}
If $M$ is finite--dimensional, 
then $H^1(\g, M) = H^2(\g, M) = 0$.
\end{thm}

As noted in \eqref{R-double-comlex:2} of the remarks at the end of Section \ref{ssec:double-complex}, we are interested in the infinite-dimensional $\g$-module  $M = \g_{\ad}^*\otimes U(\g)^{\otimes n}$. As it decomposes into a direct sum of finite-dimensional $\g$-modules, Whitehead's lemma applies to give $H^1(\g, M) = H^2(\g, M) = 0$. We will only need the vanishing of the first cohomology group, which we restate below in terms of the bicomplex of Section \ref{ssec:double-complex}.

\begin{cor}\label{cor:horizontal}
$H^1(\K^{\bullet,n}, \partial_H)\cong H^1(\g,\g_{\ad}^*\otimes U(\g)^{\otimes n})= 0$.
\end{cor}

\subsection{Proof of the claim}\label{ssec:pf-claim}
%--------------------------------------------------
We are now prepared to prove the claim appearing below \eqref{eq:eta} in the proof of Theorem \ref{thm:gens}, which will complete the proof of the theorem.  Recall from \eqref{R-double-comlex:3} of the remarks at the end of Section \ref{ssec:double-complex} that this claim is equivalent to the assertion that there 
is $\varphi\in\Hom(\g_{\ad},U(\g)) = \K^{0,1}$ satisfying 
\begin{equation}\label{claim-varphi}
\partial_H(\varphi) = \gamma\quad \text{ and }\quad \partial_V(\varphi)=\eta,
\end{equation}
 where $\gamma\in\K^{1,1}$ and $\eta\in\K^{0,2}$ are as in \eqref{eq:gamma} and \eqref{eq:eta}, respectively. Moreover, as pointed out in the same remark, the relations for $\gamma$ and $\eta$ obtained in Proposition \ref{P:gamma} are equivalent to
\begin{equation*}
\partial_H(\gamma) = 0,\qquad \partial_V(\eta)=0, \qquad \partial_V(\gamma) = \partial_H(\eta),
\end{equation*}
and $\eta = \eta^{21}$. Note that the latter identity says exactly that $\eta\in \Hom(\g_{\ad},T^2_-(U(\g))$, 
in the notation of Section \ref{ssec:coalg} above. 

Since $\partial_H(\gamma)=0$,  Corollary \ref{cor:horizontal} implies that there is a linear map $\psi:\g_{\ad}\to U(\g)$ satisfying $\gamma = \partial_H(\psi)$.
Let us define 
\begin{equation*}
\eta_1:= \eta - \partial_V(\psi):\g_{\ad}\to  U(\g)^{\otimes 2}.
\end{equation*}
Since $\partial_V(\psi)$  is given 
by $\partial_V(\psi(x)) = (\square - \Delta)(\psi(x))$, we have $\eta_1 = \eta_1^{21}$ and hence $\eta_1\in  \Hom(\g_{\ad},T^2_-(U(\g))$. Furthermore, as $\partial_H$ and $\partial_V$ commute, we have 
\begin{equation*}
\begin{aligned}
\partial_H(\eta_1) & = \partial_H(\eta) - \partial_H(\partial_V(\psi))\\
& = \partial_H(\eta) - \partial_V(\partial_H(\psi)) \\
&= \partial_H(\eta) - \partial_V(\gamma)
\end{aligned}
\end{equation*}
and hence $\partial_H(\eta_1)=0$. Note that this implies that $\eta_1$ is a $\g$-module homomorphism. Indeed, for any $\omega\in\Hom(\g_{\ad},M)$,  $\partial_H(\omega):\g\otimes\g_{\ad}
\to M$ is given by $\partial_H(\omega)(x\otimes y) = x\cdot \omega(y) - \omega([x,y])$ and hence $\partial_H(\omega)=0$
if and only if $\omega$ is a $\g$--intertwiner. Thus, we obtain
\begin{equation*}
\eta_1 \in \Hom_{\g}(\g_{\ad},T^2_-(U(\g))).
\end{equation*}
Moreover, as $\partial_V(\eta)=0$, we have $\partial_V(\eta_1) = \partial_V(\eta) - \partial_V(\partial_V(\psi)) = 0$. We are now in a position
to use Corollary \ref{cor:vertical} to conclude that there exists $\theta\in\Hom_{\g}(\g_{\ad},U(\g))$
(that is, $\partial_H(\theta)=0$)
so that $\eta_1 = \partial_V(\theta)$.
It is then immediate that $\varphi = \psi + \theta$ is the desired element of $\Hom(\g_{\ad},U(\g))$, satisfying the equations \eqref{claim-varphi}.

\section{Homogeneous quantization: relations}\label{sec:hgs-rels}
%============================================================

Our main goal in this section is to prove Theorem \ref{thm:relns}, which outputs a set of relations satisfied by the $\g$-module homomorphism $\J:\g\to \H_1$ constructed in Theorem \ref{thm:gens}, where $\H$ is an arbitrary homogeneous quantization of $(\g[u],\delta)$. These relations will be applied in Section \ref{ssec:uniqueness-thm} to obtain a presentation of $\H$ and prove the main result of this article.
\subsection{Relations of a homogeneous quantization}\label{ssec:thm-relns}
%--------------------------------------------------------

Recall from Section \ref{ssec:nu} that $\nu:\h\to U(\g)^{\h}$ is the linear map defined by the formula $\nu(h) := \frac{1}{2} \sum_{\alpha\in R_+} \alpha(h) x^-_\alpha x^+_\alpha$. 

\begin{thm}\label{thm:relns} Let $\H$ be a homogeneous quantization of $(\g[u],\delta)$, and let $\J:\g\to \H_1$ be as in Theorem \ref{thm:gens}. Then, for each $i,j\in \bfI$, one has 
\begin{align}
[\J(t_i),\J(t_j)]&=\hbar^2[\nu(t_j),\nu(t_i)] \quad  &&\text{ if }\quad \g\ncong\sl_2 \label{eq:Nsl2}\\
\left[[\J(e), \J(f)], \J(h)\right] &= \hbar^2 (f\J(e)-\J(f)e)h  \quad &&\text{ if }\quad \g\cong\sl_2 \label{eq:sl2}
\end{align}
\end{thm}
The proof of this theorem will occupy the rest of Section \ref{sec:hgs-rels}. The relation \eqref{eq:Nsl2} will be established in Sections \ref{ssec:Nsl2-1}--\ref{ssec:Nsl2-3} --- note that this still holds for $\g\cong\sl_2$, but reduces to a trivial relation. The identity \eqref{eq:sl2} is then established in Sections \ref{ssec:sl2-1} and \ref{ssec:sl2-2}, under the hypothesis that $\g\cong\sl_2$.

\subsection{Proof of \eqref{eq:Nsl2} I}\label{ssec:Nsl2-1}
%--------------------------------------------------------
We begin by showing that the difference between the left and right-hand sides of \eqref{eq:Nsl2} is of the form $\hbar^2 g(t_i,t_j)$ for an antisymmetric function $g$ on $\h\otimes \h$ with values in $\h$. 
\begin{prop}\label{P:gamma}
There is an antisymmetric linear map $g:\h\otimes \h\to \h$ satisfying
\begin{equation*}
\BB{\J(t_i),\J(t_j)} + \hbar^2 [\nu(t_i),\nu(t_j)] = \hbar^2 g(t_i,t_j) \quad \forall\quad  i,j\in \bfI.
\end{equation*}
\end{prop}

\begin{proof}
We claim that $\BB{\J(t_i),\J(t_j)} + \hbar^2 [\nu(t_i),\nu(t_j)]$ is a primitive element in $\H$.
In view of Proposition \ref{pr:prim}, this will imply that it belongs to $\hbar^2\g$. As it also commutes with $\h$ and 
is antisymmetric in $i$ and $j$, it must therefore be of the form $\hbar^2 g(t_i,t_j)$ for some
$g:\h\wedge \h \to \h$, as desired. To establish the claim, we apply Theorem \ref{thm:gens} and that $\Delta$ is an algebra homomorphism to obtain
\begin{equation*}
\begin{aligned}
\Delta([\J(t_i),\J(t_j)]) ={}& \BB{\Delta(\J(t_i)),\Delta(\J(t_j))}\\
={}&
\BB{
\square{\J(t_i)}+\frac{\hbar}{2}[t_i\otimes 1,\Omega],\square{\J(t_j)}+\frac{\hbar}{2}[t_j\otimes 1,\Omega]
}
\\
={}&\square([\J(t_i),\J(t_j)])+\frac{\hbar^2}{4}\left[ [t_i\otimes 1,\Omega], [t_j\otimes 1,\Omega]\right]\\
&\qquad +\frac{\hbar}{2}\Big([\square\J(t_i),[t_j\otimes 1,\Omega]]-  [\square\J(t_j),[t_i\otimes 1,\Omega]]\Big).
\end{aligned} 
\end{equation*}
Since $\J$ is a $\g$-module homomorphism, the term $[\square\J(t_i),[t_j\otimes 1,\Omega]]$ is symmetric in $i$ and $j$. Indeed, we have
\begin{equation*}
[\square\J(t_i),[t_j\otimes 1,\Omega]]
=\sum_{\alpha\in \Delta^+}\alpha(t_i)\alpha(t_j)(\J\otimes \mathrm{Id} - \mathrm{Id}\otimes \J) (x_\alpha^+\otimes x_\alpha^- + x_\alpha^-\otimes x_\alpha^+).
\end{equation*}
Therefore, the expression for $\Delta([\J(t_i),\J(t_j)])$ obtained above reduces to
\begin{equation*}
\Delta([\J(t_i),\J(t_j)])=\square([\J(t_i),\J(t_j)])+\frac{\hbar^2}{4}\left[ [t_i\otimes 1,\Omega], [t_j\otimes 1,\Omega]\right].
\end{equation*}
On the other hand, by Part \eqref{GNW:3} of Lemma \ref{lem:GNW}, we have
\begin{equation*}
\hbar^2 \Delta([\nu(t_i),\nu(t_j)]) = \hbar^2 \square([\nu(t_i),\nu(t_j)]) -  \frac{\hbar^2}{4} \left[ [t_i\otimes 1,\Omega], [t_j\otimes 1,\Omega]\right],
\end{equation*}
and hence $[\J(t_i),\J(t_j)]+\hbar^2 [\nu(t_i),\nu(t_j)]$ is indeed a primitive element of $\H$, as claimed. \qedhere
\end{proof}

\subsection{Proof of \eqref{eq:Nsl2} II}\label{ssec:Nsl2-2}
%-----------------------------------------------------------

Our next step is to show the linear map $g: \h\wedge\h\to \h$ of Proposition \ref{P:gamma} is identically zero. To this end, let us introduce an auxiliary linear map $\T:\h\to \H_1$ by the formula 
\begin{equation*}
\T(h):=\J(h)-\hbar \nu(h) \quad \forall \; h\in\h.
\end{equation*}
The main identity of Proposition \ref{P:gamma} may then be expressed equivalently as
\begin{equation}\label{gamma-T}
[\T(t_i),\T(t_j)]=\hbar^2 g(t_i,t_j) \quad \forall \quad i,j \in \bfI.
\end{equation} 
Indeed, this follows from the observation that the bracket $[\J(t_i),\nu(t_j)]$ is symmetric in $i$ and $j$. Explicitly, one has
\begin{equation*}
[\J(t_i),\nu(t_j)]=\sum_{\alpha\in \Delta^+} \alpha(t_i)\alpha(t_j)\left( x_\alpha^- \J(x_\alpha^+) - \J(x_\alpha^- )x_\alpha^+\right)= [\J(t_j),\nu(t_i)].
\end{equation*}
Due to \eqref{gamma-T}, our task is reduced to showing that $[\T(t_i),\T(t_j)]=0$ for all $i,j\in \bfI$. To achieve this, we will make use of the following lemma. 
\begin{lem}\label{L:T}
For each $i\in \bfI$, define $x_{i,1}^\pm,\xi_{i,1}\in \H_1$ by
\begin{equation*}
x_{i,1}^\pm:=\pm\frac{1}{(\alpha_i,\alpha_i)}[\T(t_i),x_i^\pm] \quad \text{ and }\quad \xi_{i,1}:=\T(t_i)+\frac{\hbar}{2} t_i^2.
\end{equation*}
Then, for each $h\in \h$ and $i,j\in \bfI$, we have
\begin{equation*}
[\T(h),x_{i}^\pm]=\pm \alpha_i(h) x_{i,1}^\pm \quad \text{ and }\quad [x_{i,1}^+,x_j^-]=\delta_{i,j}\xi_{i,1}=[x_i^+,x_{j,1}^-].
\end{equation*}
\end{lem}
\begin{proof}
As $\J$ is a $\g$-module homomorphism, we have
\begin{equation*}
[\T(h),x_i^\pm]=[\J(h),x_i^\pm]-\hbar[\nu(h),x_i^\pm]= \frac{\alpha_i(h)}{(\alpha_i,\alpha_i)}[\J(t_i),x_i^\pm]-\hbar [\nu(h),x_i^\pm].
\end{equation*}
By Part \eqref{GNW:1} of Lemma \ref{lem:GNW}, $[\nu(h),x_i^\pm]=\pm \alpha_i(h)w_i^\pm$, where $w_i^\pm=\pm \frac{1}{(\alpha_i,\alpha_i)}[\nu(t_i),x_i^\pm]$, 
so the above coincides with $\pm \alpha_i(h) x_{i,1}^\pm$.
Similarly, since $x_{i,1}^+ = \J(x_i^+)-\hbar w_i^+$, we have
\begin{equation*}
[x_{i,1}^+,x_j^-]= \J([x_i^+,x_j^-])-\hbar [w_i^+,x_j^-]=\delta_{i,j} \J(t_i) - \hbar [w_i^+,x_j^-].
\end{equation*}
By Part \eqref{GNW:2} of Lemma \ref{lem:GNW},  $[w_i^+,x_j^-]=\delta_{i,j}(\nu(t_i)-\frac{\hbar}{2}t_i^2)$ and hence the above becomes
\begin{equation*}
[x_{i,1}^+,x_j^-]=\delta_{i,j} (\T(t_i)+\frac{\hbar}{2}t_i^2)=\delta_{i,j} \xi_{i,1}.
\end{equation*}
Finally, applying $\mathrm{ad}(\T(t_j))$ to $[x_i^+,x_j^-]=\delta_{i,j}t_i$ gives 
\begin{equation*}
(\alpha_i,\alpha_j)[x_{i,1}^+,x_j^-]=(\alpha_j,\alpha_j)[x_i^+,x_{j,1}^-]
\end{equation*}
and therefore we must also have $[x_i^+,x_{j,1}^-]=\delta_{i,j} \xi_{i,1}$. \qedhere
\end{proof}

\subsection{Proof of \eqref{eq:Nsl2} III}\label{ssec:Nsl2-3}
%----------------------------------------------------------

With \eqref{gamma-T} and Lemma \ref{L:T} at our disposal, we are now prepared to complete the proof of \eqref{eq:Nsl2} by proving that linear map $g:\h\wedge \h\to \h$ of Proposition \ref{P:gamma} vanishes.

\begin{prop}\label{T:T(h)-commute}
For each $i, j \in \bfI$, we have $[\T(t_i),\T(t_j)]=\hbar^2 g(t_i,t_j)=0$, and thus $g$ is identically zero.
\end{prop}

\begin{proof}
By \eqref{gamma-T} we have $[\T(t_i),\T(t_j)]=\hbar^2 g(t_i,t_j)$ for all $i,j\in \bfI$, and hence the equality from the statement of the proposition trivially holds when $i=j$. Suppose instead that $i,j\in \bfI$ are such that $i\neq j$.  
Then by applying $\mathrm{ad}(x_i^+)$ to both sides of \eqref{gamma-T} while using the first identity of Lemma \ref{L:T}, we obtain
\begin{equation*}
-\alpha_i(t_i)[x_{i,1}^+, \T(t_j)]-\alpha_i(t_j)[\T(t_i),x_{i,1}^+]=-\hbar^2 \alpha_i\!\left(g(t_i,t_j)\right) x_i^+,
\end{equation*}
which is equivalent to
\begin{equation*}
[\T( \alpha_i(t_j)t_i-\alpha_i(t_i)t_j), x_{i,1}^+]=\hbar^2 \alpha_i\!\left(g(t_i,t_j)\right) x_i^+.
\end{equation*}
Since $\alpha_i(t_j)t_i-\alpha_i(t_i)t_j\in \mathrm{Ker}(\alpha_i)$, applying $-\mathrm{ad}(x_i^-)$ to both sides of this identity yields
\begin{equation*}
[\T( \alpha_i(t_j)t_i-\alpha_i(t_i)t_j), \xi_{i,1}]=\hbar^2 \alpha_i\!\left(g(t_i,t_j)\right) t_i.
\end{equation*}
As $\xi_{i,1}=\T(t_i)+\frac{\hbar}{2} t_i^2$ and $t_i^2$ commutes with $\T(\h)$, the above is equivalent to
\begin{equation*}
\alpha_i(t_j)g(t_i,t_i)-\alpha_i(t_i)g(t_j,t_i)=\alpha_i\!\left(g(t_i,t_j)\right) t_i. 
\end{equation*}
Since  $g(t_i,t_i)=0$, this is equivalent to 
\begin{equation*}
g(t_j,t_i) = -\frac{\alpha_i\!\left(g(t_i,t_j)\right)}{\alpha_i(t_i)}t_i \in \mathbb{C}t_i.
\end{equation*}
As we also have $g(t_j,t_i) = -g(t_i,t_j) \in \mathbb{C}t_j$, this is only possible if $g(t_i,t_j) = 0$. \qedhere
\end{proof}

\subsection{Proof of \eqref{eq:sl2} I}\label{ssec:sl2-1}
%------------------------------------------------------
We now shift our focus to establishing the identity \eqref{eq:sl2} of Theorem \ref{thm:relns}. Throughout Sections \ref{ssec:sl2-1} and \ref{ssec:sl2-2}, we assume that $\g=\Lsl_2$. In particular, the Casimir tensor $\Omega$ takes the form  
\begin{equation*}
\Omega = \frac{1}{2} h\otimes h + e\otimes f + f\otimes e
\end{equation*}
and thus, by Theorem \ref{thm:gens}, we have 
\begin{equation}\label{eq:Del-Jefh}
\begin{aligned}
(\Delta-\square)(\J(h)) &= \hbar (e\otimes f - f\otimes e) \\
(\Delta-\square)(\J(e)) &= \frac{\hbar}{2} (h\otimes e - e\otimes h) \\
(\Delta-\square)(\J(f)) &=\frac{\hbar}{2} (f\otimes h-h\otimes f) \\
\end{aligned}
\end{equation}
Our first main step towards establishing \eqref{eq:Nsl2} is to prove the following analogue of Proposition \ref{P:gamma}. 
\begin{prop}\label{pr:sl2-1}
There exists $a\in\h=\C h$ such that
\[
\BB{\BB{\J(e),\J(f)},\J(h)} - \hbar^2 (f\J(e)-\J(f)e)h  = \hbar^3a.
\]
\end{prop}

\begin{proof}
Again, we will show that the left-hand side is primitive. Since it of weight zero, it will then follow automatically from Proposition \ref{pr:prim} that it is of the form $\hbar^3 a$ for some $a\in \h$.
We break this computation into four main steps.

\noindent {\bf Step 1.}  $\J(e)$ and $\J(f)$ satisfy the identity 
\begin{equation*}
(\Delta-\square)\BB{\J(e),\J(f)} = \hbar (\J\otimes\Id+\Id\otimes\J)(e\otimes f - f\otimes e) 
-\frac{\hbar^2}{2}\square(h)\Omega.
\end{equation*}

To establish this identity, we first apply the formulas \eqref{eq:Del-Jefh} to obtain
\[
\begin{aligned}
(\Delta-\square)\BB{\J(e),\J(f)} ={}&\frac{\hbar}{2} \lp
\BB{h\otimes f - f\otimes h, \square(\J(e))} + 
\BB{h\otimes e - e\otimes h, \square(\J(f))}
\rp \\
&-\frac{\hbar^2}{4} \BB{h\otimes e - e\otimes h, h\otimes f - f\otimes h}.
\end{aligned}
\]
Now, we use the fact that $\J$ is an $\Lsl_2$-intertwiner to get
\[
\begin{aligned}
\BB{h\otimes f - f\otimes h, \square(\J(e))} &= 2\J(e)\otimes f - 2 f\otimes \J(e) - h\otimes \J(h) + \J(h)\otimes h, \\
\BB{h\otimes e - e\otimes h, \square(\J(f))} &= 2e\otimes \J(f) - 2\J(f)\otimes e  + h\otimes \J(h) - \J(h)\otimes h.
\end{aligned}
\]
Therefore, the coefficient of $\hbar$ in the above expression for $(\Delta-\square)\BB{\J(e),\J(f)}$ is $(\J\otimes\Id+\Id\otimes\J)(e\otimes f - f\otimes e)$, as claimed.
To see that the coefficient of $\hbar^2$ is $(-1/2)\square(h)\Omega$, observe that
\begin{align*}
[h\otimes e - e\otimes h, h\otimes f - f\otimes h] ={}& h^2\otimes h + h\otimes h^2 -hf\otimes eh + fh\otimes he \\ 
&- eh\otimes hf + he\otimes fh \\
={}& h^2\otimes h + h\otimes h^2 + 2hf\otimes e + 2f\otimes he \\
&+ 2e\otimes hf + 2he\otimes f \\
={}& 2 \square(h)\Omega,
\end{align*}
which outputs the desired result and completes our first step.

\medskip

\noindent {\bf Step 2.} Let $\mathcal{A} = \BB{\BB{\J(e),\J(f)},\J(h)}$. Then $\mathcal{A}$ satisfies
\begin{equation*}
(\Delta-\square)(\mathcal{A}) = \hbar^2 \mathcal{C} +\frac{\hbar^3}{4}\square(h)\BB{\BB{h\otimes 1,\Omega},\Omega},
\end{equation*}
where $\mathcal{C}$ is given by the formula 
\begin{equation*}
\mathcal{C} ={} \frac{1}{2}\square(h)\BB{\square(\J(h)),\Omega} + \BB{(\J\otimes\Id+\Id\otimes\J)(e\otimes f - f\otimes e), e\otimes f - f\otimes e}.
%\\
%C_2 ={}& -\frac{1}{4}\square(h)\BB{\BB{h\otimes 1,\Omega},\Omega} 
\end{equation*}

This step of our proof follows from the identity obtained in Step 1, together with the formula  $(\Delta-\square)(\J(h)) = \hbar (e\otimes f - f\otimes e) = \frac{\hbar}{2}[h\otimes 1,\Omega]$ 
from \eqref{eq:Del-Jefh}, and the fact that $\square(h)$ commutes both with $\square(\J(h))$ and $\Omega$. In more detail, using these results we obtain
\begin{equation*}
\Delta(\mathcal{A})=\square(\mathcal{A})+ \hbar \mathcal{C}_0+\hbar^2 \mathcal{C}  +\frac{\hbar^3}{4}\square(h)\BB{\BB{h\otimes 1,\Omega},\Omega},
\end{equation*}
where $\mathcal{C}_0$ is the element
\begin{align*}
\mathcal{C}_0 ={}& \BB{\lp\J\otimes\Id+\Id\otimes\J\rp(e\otimes f - f\otimes e), \square(\J(h))}\\
&+ \BB{\square\lp\BB{\J(e),\J(f)}\rp, e\otimes f - f\otimes e}.
\end{align*}
Thus, it remains to see that $\mathcal{C}_0=0$. This is a consequence of the following easy to verify identities:
\begin{equation*}
\begin{aligned}
\BB{\J(e)\otimes f + e\otimes\J(f),\square(\J(h))} &= \BB{\J(e),\J(h)}\otimes f + e\otimes \BB{\J(f),\J(h)} \\
\BB{\J(f)\otimes e + f\otimes\J(e),\square(\J(h))} &= \BB{\J(f),\J(h)}\otimes e + f\otimes \BB{\J(e),\J(h)} \\
\BB{\BB{\J(e),\J(f)}\otimes 1, e\otimes f-f\otimes e} &= \BB{\J(h),\J(e)}\otimes f + \BB{\J(f),\J(h)}\otimes e \\
\BB{1\otimes\BB{\J(e),\J(f)}, e\otimes f-f\otimes e} &= f\otimes \BB{\J(e),\J(h)} + e\otimes \BB{\J(h),\J(f)}
\end{aligned}
\end{equation*}

\medskip

\noindent {\bf Step 3.} Let $\mathcal{B} = (f\J(e)-\J(f)e)h$. Then $\mathcal{B}$ satisfies
\begin{equation}\label{eq:pf-sl2-3}
(\Delta-\square)(\mathcal{B}) = \frac{1}{2}\BB{\square(\J(h)),\mathcal{L}} + \frac{\hbar}{4}\square(h)
\BB{\BB{h\otimes 1,\Omega},\Omega},
\end{equation}
where $\mathcal{L} = (\Delta-\square)(feh)$. 

First observe that $\mathcal{B}$ can be expressed as $
\mathcal{B} = \frac{1}{2} \BB{\J(h),feh}$. Therefore, applying $\Delta$ to $\mathcal{B}$ while using that $\Delta(feh)=\square(feh)+\mathcal{L}$ yields the formula
\begin{align*}
\Delta(\mathcal{B}) &= \frac{1}{2} \BB{\square(\J(h))+\frac{\hbar}{2}[h\otimes 1,\Omega],\Delta(feh)} \\
&= \square(\mathcal{B}) + \frac{1}{2}\BB{\square(\J(h)),\mathcal{L}} + \frac{\hbar}{4}\BB{\BB{h\otimes 1,\Omega},\Delta(feh)}.
\end{align*}
Now let us write $\Delta(feh) = \Delta(fe)\square(h)$. Since both $\square(h)$ and $\Delta(fe)$ commute with $\Omega$, and $\square(h)$ also commutes with $h\otimes 1$, we have
\begin{align*}
\BB{\BB{h\otimes 1,\Omega},\Delta(feh)}
&= \BB{\BB{h\otimes 1,\Omega},\Delta(fe)}\square(h)=[[h\otimes 1,\Delta(fe)],\Omega]\square(h).
\end{align*}
Therefore, Step 3 follows from the observation that
$[h\otimes 1,\Delta(fe)] = [h\otimes 1,\Omega]$. 

\medskip 

\noindent {\bf Step 4.} The elements $\mathcal{A}$ and $\mathcal{B}$ from Steps 2 and 3 satisfy $(\Delta-\square)(\mathcal{A} - \hbar^2\mathcal{B}) = 0$. Equivalently, one has 
\begin{equation*}
\BB{\BB{\J(e),\J(f)},\J(h)} - \hbar^2 (f\J(e)-e\J(f))h \in \operatorname{Prim}(\H).
\end{equation*}

To prove this, we combine the identities from Steps 2 and 3 to obtain
\begin{equation*}
(\Delta-\square)(\mathcal{A} - \hbar^2\mathcal{B}) = \hbar^2 \lp \mathcal{C}- \frac{1}{2}\BB{\square(\J(h)),\mathcal{L}}\rp,
\end{equation*}
where $\mathcal{C}$ and $\mathcal{L}$ are as in the statements of Steps 2 and 3, respectively. Next, observe that we can rewrite $\mathcal{L}$ as 
\begin{align*}
\mathcal{L} &= \Delta(feh) - \square(feh) \\
&= \Delta(fe)\square(h) - \square(fe)\square(h) + fe\otimes h + h\otimes fe\\
&= (e\otimes f + f\otimes e)\square(h) + fe\otimes h + h\otimes fe ,
\end{align*}
which gives  $\BB{\square(\J(h)),\mathcal{L}} = \BB{\square(\J(h)),\Omega\square(h) + fe\otimes h + h\otimes fe}$. Therefore, by definition of the element $\mathcal{C}$, we have 
\begin{align*}
\mathcal{C} - \frac{1}{2}\BB{\square(\J(h)),\mathcal{L}} ={}& \BB{\lp\J\otimes\Id+\Id\otimes\J\rp(e\otimes f - f\otimes e),
e\otimes f - f\otimes e} \\
&-\frac{1}{2} \BB{\square(\J(h)),h\otimes fe + fe\otimes h}
\end{align*}
which is easily seen to be zero, by a direct verification. This completes Step 4, and thus the proof of Proposition \ref{pr:sl2-1}. 
\end{proof}

\subsection{Proof of \eqref{eq:sl2} II}\label{ssec:sl2-2}
%------------------------------------------------------
Armed with Proposition \ref{pr:sl2-1}, we are now prepared to complete the proof of Theorem \ref{thm:relns} by establishing the relation \eqref{eq:sl2}. Recall that we are given a homogeneous quantization $\H$ of  $(\Lsl_2[u],\delta)$, equipped with $\J:\sl_2\to \H_1$ as in Theorem \ref{thm:gens}. By Proposition \ref{pr:sl2-1}, there is an element $a\in \C h$ satisfying
\[
\BB{\BB{\J(e),\J(f)},\J(h)} - \frac{\hbar^2}{2} \BB{\J(h),feh} =\hbar^3 a.
\]
We wish to show that $a=0$. To this end, observe that $\C h\subset \sl_2 $ is contained in the eigenspace for $T=\ad(f)\circ\ad(e)$ with eigenvalue $2$, and hence $T(a)=2a$. 
We will now verify that if $a\neq 0$, then it is an eigenvector for this same operator with eigenvalue $6$, which is impossible. We begin by computing the action of $T$ on $\BB{\BB{\J(e),\J(f)},\J(h)}$:
\begin{align*}
T\left(\BB{\BB{\J(e),\J(f)},\J(h)}\right)
={}& 
\big[f, \BB{\BB{\J(e),\J(h)},\J(h)}
-2\BB{\BB{\J(e),\J(f)},\J(e)} \big]\\
={}&
4\BB{\BB{\J(e),\J(f)},\J(h)}+2\BB{\BB{\J(e),\J(h)},\J(f)}\\
&+2\BB{\BB{\J(h),\J(f)},\J(e)} +6\BB{\BB{\J(e),\J(f)},\J(h)}\\
={}& 
6\BB{\BB{\J(e),\J(f)},\J(h)}.
\end{align*}
In order to apply $T$ to $\BB{\J(h),feh}$, it is convenient to replace
$fe$ by $\kappa$, where $\kappa$ is half of the quadratic central Casimir element in $U(\sl_2)$:
\begin{equation*}
\kappa = \frac{1}{4}h^2 + \frac{1}{2}h + fe.
\end{equation*}
Since $\J(h)$ commutes with $h$, we then have $[\J(h),feh]=[\J(h),\kappa]h$. Applying $\ad(e)$ to this element then gives
\begin{equation*}
\ad(e)(\BB{\J(h),\kappa}h) = -2\BB{\J(e),\kappa}h -2\BB{\J(h),\kappa}e.
\end{equation*}
Applying now $\ad(f)$ to this, we find that
\begin{equation*}
T(\BB{\J(h),\kappa}h) = 4\BB{\J(h),\kappa}h - 4 \BB{\J(e)f + \J(f)e,\kappa}.
\end{equation*}
Since $\J(e)f + \J(f)e$ is of weight zero, we may replace $\kappa$ by $fe$ in the last bracket. That $T(\BB{\J(h),\kappa}h)=6\BB{\J(h),\kappa}h$ then follows immediately from the identity
\begin{equation*}
\BB{\J(e)f + \J(f)e, fe} = (\J(f)e-f\J(e))h = -\frac{1}{2} \BB{\J(h),\kappa}h
\end{equation*}
which is easily verified directly. Thus, we have proven that $T(a)=  6a$, and hence that $a=0$.  This completes the proof of \eqref{eq:sl2}.	

\section{Uniqueness of homogeneous quantization}\label{sec:hgs-unique}
%==================================================================

In this section, we combine Theorem \ref{thm:current-min} with the main results of Sections \ref{sec:hgs-gens} and \ref{sec:hgs-rels} --- Theorems \ref{thm:gens} and \ref{thm:relns} --- in order to obtain a proof of Theorem \ref{thm:main}, which is reformulated as Theorem \ref{thm:uniqueness} below. We conclude in Section \ref{ssec:existence}  by providing a brief literature review, which connects our results to the theory of Yangians by expanding on the remarks made in Section \ref{ssec:intro-remarks}. 

\subsection{The graded algebra \texorpdfstring{$\U$}{U}}\label{ssec:yangian}
%--------------------------------

Let $\U$ be the unital associative $\C[\hbar]$-algebra generated by  $\{\iota(x),J(x)\}_{x\in \g}$, with the  following defining relations:
\begin{enumerate}\setlength{\itemsep}{5pt}
\item\label{Y-min:1} $\iota: \g \to \U$ is a Lie algebra map and $J: \g \to \U$ is a $\g$-module homomorphism.
That is, for all $x,y\in\g$ and $\lambda,\mu\in\C$, we have:
\[
\begin{aligned}
\iota(\lambda x + \mu y) = \lambda \iota(x) + \mu \iota(y), &\quad& \iota([x,y]) = [\iota(x),\iota(y)], \\
J(\lambda x + \mu y) = \lambda J(x) + \mu J(y), &\quad& J([x,y]) = [\iota(x),J(y)].
\end{aligned}
\]
In particular, we may extend $\iota$ to a $\C$-algebra homomorphism $U(\g)\to \U$. 
\item\label{Y-min:2} For each $i,j\in \bfI$, one has
\begin{align*}
[J(t_i),J(t_j)]&=\hbar^2\iota([\nu(t_j),\nu(t_i)])  &&\text{ if }\quad \g\ncong\sl_2 \\
\left[[J(e), J(f)], J(h)\right] &= \hbar^2 (\iota(f)J(e)-J(f)\iota(e))\iota(h)  &&\text{ if }\quad \g\cong\sl_2 
\end{align*}
\end{enumerate}

It follows from this definition that $\U$ is an $\N$-graded algebra over the graded ring $\C[\hbar]$, with grading determined by
 $\deg \iota(x)=0$ and $\deg J(x)=1$, for all $x\in \g$. Furthermore, by Theorem \ref{thm:current-min},  the assignment
\begin{equation}\label{U/hU=U(g[u])}
 \iota(x)\mapsto x \quad \text{ and }\quad J(x)\mapsto G(x)=xu \quad \forall\;x\in \g
\end{equation}
extends to a surjective homomorphism of graded $\C[\hbar]$-algebras $\U\twoheadrightarrow U(\g[u])$ with kernel $\hbar \U$, where  $\hbar$ is understood to act trivially on $U(\g[u])$. Thus, $\U$ is an $\N$-graded algebra deformation of $U(\g[u])$ over $\C[\hbar]$.

\subsection{The uniqueness theorem}\label{ssec:uniqueness-thm}
%--------------------------------------------------------
The following theorem,  coupled with the formulas for $\Delta(\J(x))$ from Theorem \ref{thm:gens} and the definition of $\U$ given in Section \ref{ssec:yangian}, immediately implies Theorem \ref{thm:main}.
\begin{thm}\label{thm:uniqueness}
Let $\H$ be a homogeneous quantization of $(\g[u],\delta)$. Then $\H$ is isomorphic to $\U$ as an $\N$-graded $\C[\hbar]$-algebra. An isomorphism $\Psi_\H:\U\iso \H$ is given by  
\begin{equation*}
\Psi_\H(\iota(x))=x \quad \text{ and }\quad \Psi_\H(J(x))=\J(x) \quad \forall\; x\in \g,
\end{equation*}
where  $\J:\g\to \H_1$ is as in Theorem \ref{thm:gens}.  
 Moreover, if $\H^{\prime}$ is any other homogeneous quantization of $(\g[u],\delta)$, then the composition 
\begin{equation*}
\Psi_{\H^\prime}\circ \Psi_\H^{-1}:\H \iso \H^{\prime}
\end{equation*}
is an isomorphism of graded Hopf algebras over $\C[\hbar]$.
\end{thm}

\begin{proof} The second assertion of the theorem, concerning  the composition $\Psi_{\H^\prime}\circ \Psi_\H^{-1}$, is a consequence of the first assertion and the formulas for the coproduct $\Delta$, counit $\varepsilon$, and antipode $S$ of $\H$ given in Theorem \ref{thm:gens} and Remark \ref{rem:Hopf}. It is thus sufficient to establish that first statement. 

By Theorems \ref{thm:gens} and \ref{thm:relns}, the elements $\{x,\J(x)\}_{x\in \g}\subset \H$ satisfy the defining relations \eqref{Y-min:1} and \eqref{Y-min:2} of $\U$ from Section \ref{ssec:yangian}. Therefore, there is a graded algebra homomorphism 
\begin{equation*}
\Psi_\H:\U\to \H
\end{equation*}
 given as in the statement of the theorem. Since $\U$ and $\H$ are $\N$-graded $\C[\hbar]$-modules and $\H$ is torsion free, to see that $\Psi_\H$ is an isomorphism it suffices to prove that its classical limit $\psi_\H:\U/\hbar\U\to \H/\hbar \H$, obtained by reducing modulo $\hbar$, is an isomorphism.  
Using the identifications of $\H/\hbar \H$ and $\U/\hbar\U$ with $U(\g[u])$ provided via \eqref{N-quant:2} of Definition \ref{ssec:hgs-q-defn} and \eqref{U/hU=U(g[u])}, respectively, we may view $\psi_\H$ as an algebra homomorphism
\begin{equation*}
\psi_\H:U(\g[u])\to U(\g[u]).
\end{equation*}
Since the images of $\J(x)$ and $J(x)$  in $U(\g[u])$ both coincide with $G(x)=xu\in \g[u]$ and $\psi_\H$ is the identity on $\g$, we have $\psi_\H=\Id_{U(\g[u])}$. This proves that $\Psi_\H$ is an isomorphism. \qedhere
\end{proof}

\subsection{Concluding remarks}\label{ssec:existence}
%--------------------------------------------------------
  We now give a number of remarks closely related to the existence portion of Drinfeld's Theorem \ref{thm:Dr}, which expand upon the comments from Section \ref{ssec:intro-remarks} and serve to position our results within the theory of Yangians. 
  
  Suppose first that $\g \ncong \sl_2$. Then, as noted in Section \ref{ssec:intro-remarks}, the graded algebra $\U$ introduced in Section \ref{ssec:yangian} is precisely the Yangian $\Yhg$ of the simple Lie algebra $\g$, expressed in the presentation obtained in Theorem 2.13 of \cite{GNW}. This theorem of \cite{GNW} improves upon the presentation for $\Yhg$ obtained by Levendorskii in \cite[Thm.~1.2]{Lev-Gen}, by showing that one can omit its degree 3 defining relation \cite[(1.6)]{Lev-Gen}, provided $\g$ is of rank at least two.  Here we note that the results of \cite{GNW,Lev-Gen} are written in terms of the degree $1$ elements 
\begin{equation*}
\tilde{h}_{i,1}=J(t_i)-\hbar \iota(\nu(t_i)) \quad \text{ and }\quad x_{i,1}^\pm =\pm \frac{1}{(\alpha_i,\alpha_i)}[\tilde{h}_{i,1},\iota(x_i^\pm)],
\end{equation*}
which play the roles of $\T(t_i)$ and $x_{i,1}^\pm$ from Lemma \ref{L:T}, rather than the $\g$-module homomorphism $J$ from the definition of $\U$.
If instead  $\g\cong \sl_2$, then $\U$ is the Yangian $Y_\hbar(\sl_2)$ expressed in a well-know variation of Levendorskii's realization \cite{Lev-Gen}, using the elements $J(e)$, $J(f)$ and $J(h)$; see \cite[\S2.9]{Mo-Handbook}, for instance. We refer the reader to \cite[\S A]{GRWEquiv} for a detailed proof of the equivalence of the relation from \eqref{Y-min:2} in the Definition of $\U$ and the degree $3$ relation \cite[(1.6)]{Lev-Gen}.

An important consequence of these identification of $\U$ with $\Yhg$, and established results from the theory of Yangians, is  that $\U$ is in fact a homogeneous quantization of $(\g[u],\delta)$.
%with $J:\g\to \U$ satisfying the conditions of $\J$ from Theorem \ref{thm:gens}. 
Although this fact has not played a direct role in our results, it is worth highlighting a few crucial ingredients that make up its proof:
\begin{enumerate}\setlength{\parskip}{3pt}
\item The main obstruction to establishing that $\U$ is a graded Hopf algebra lies in showing that the formulas 
\begin{equation*}
\Delta(\iota(x))=\square(\iota(x)) \quad \text{ and }\quad \Delta(J(x))=\square(J(x))+\frac{\hbar}{2}[\iota(x)\otimes 1,\Omega_\iota],
\end{equation*}
for all $x\in\g$, 
determine a  $\C[\hbar]$-algebra homomorphism $\Delta:\U\to \U\otimes \U$, where $\Omega_\iota=(\iota\otimes \iota)(\Omega)$. For $\g\ncong \sl_2$, a proof of this was given in Theorem 4.9 of \cite{GNW}; see also Remark 4.10 and Proposition 4.24 therein.

 For $\g\cong\sl_2$, it is unclear to the authors if a direct proof of this fact has appeared in the literature, using the defining relations of $\U$. That being said, it is straightforward to see that the defining relations \eqref{Y-min:1} for $\U$ are preserved by $\Delta$. That the relevant relation of \eqref{Y-min:2} is also preserved is in fact a consequence of the computations carried out in the proof of Proposition \ref{pr:sl2-1}. Indeed, these computations show that the element
\begin{align*}
[[\Delta(J(e)), &\Delta(J(f))], \Delta(J(h))]\\
&-\hbar^2 \big(\Delta(\iota(f))\Delta(J(e))-\Delta(J(f))\Delta(\iota(e))\big)\Delta(\iota(h))
\end{align*}
coincides with 
\begin{equation*}
\square\big(\left[[J(e), J(f)], J(h)\right]-\hbar^2 (\iota(f)J(e)-J(f)\iota(e))\iota(h)\big),
\end{equation*}
which vanishes as a consequence of \eqref{Y-min:2}. 
\item Given the above, it follows by Remark \ref{rem:Hopf} that the counit $\varepsilon$ and antipode $S$ of $\U$ are uniquely determined by the formulas
\begin{gather*}
\varepsilon(\iota(x))=\varepsilon(J(x))=0, \quad S(\iota(x))=-\iota(x), \quad S(J(x))=-J(x)+\frac{\hbar}{4}c_\g \iota(x),
\end{gather*}
for all $x\in \g$, where $c_\g$ is the eigenvalue of quadratic Casimir element $C\in U(\g)$ on the adjoint representation of $\g$, as in Remark \ref{rem:Hopf}.  It follows easily from the definition of $\U$ that these formulas indeed define a graded algebra homomorphism $\varepsilon:\U\to \C[\hbar]$ and anti-algebra homomorphism $S:\U\to \U$. By construction, they satisfy the counit and antipode anxioms of a Hopf algebra. 

\item Finally, it follows from the above remarks and the formulas for $\delta(x)$ and $\delta(G(x))$ given in Section \ref{ssec:current-min} that $\U$ will satisfy all of the conditions of Definition \ref{ssec:hgs-q-defn} provided it is torsion free as a $\C[\hbar]$-module.  That the Yangian $\Yhg$ has this property is a consequence of its Poincar\'{e}--Birkhoff--Witt theorem, which was first established by Levendorskii in \cite{LevPBW} in Drinfeld's \textit{new} or \textit{loop} presentation \cite{DrNew}; see also Proposition 2.2 of \cite{GRWEquiv} and Theorem B.2 of \cite{FiTs19}. Since this presentation of $\Yhg$ is isomorphic to $\U$ by the results of \cite{GNW} and \cite{Lev-Gen}, we can conclude that $\U$ is a homogeneous quantization of $(\g[u],\delta)$. 
\end{enumerate}

\bibliographystyle{amsalpha}
\bibliography{Yangians}

\end{document}